\documentclass[a4paper,reqno,11pt]{amsart}

\newfont{\cyr}{wncyr10 scaled 1100}

\usepackage[left=2.7cm,right=2.7cm,top=3.5cm,bottom=3cm]{geometry}

\usepackage{amsthm,amssymb,amsmath,amsfonts,mathrsfs,amscd}
\usepackage{mathtools}
\usepackage{extarrows}
\usepackage[latin1]{inputenc}
\usepackage[all]{xy}
\usepackage{latexsym}
\usepackage{longtable}
\usepackage{xcolor}
\usepackage{comment}
\usepackage[shortlabels]{enumitem}
\usepackage{setspace}
\setdisplayskipstretch{}
\usepackage{multirow}
\usepackage{xcolor,colortbl}
\usepackage{changepage}
\usepackage{float}
\theoremstyle{plain}
\newtheorem{theorem}{Theorem}[section]
\newtheorem{corollary}[theorem]{Corollary}
\newtheorem{lemma}[theorem]{Lemma}
\newtheorem{proposition}[theorem]{Proposition}

\theoremstyle{definition}
\newtheorem{definition}[theorem]{Definition}

\newtheorem{examplewr}[theorem]{Example}

\theoremstyle{remark}
\newtheorem{obswr}[theorem]{Observation}
\newtheorem{remarkwr}[theorem]{Remark}

\definecolor{Gray}{gray}{0.85}
\definecolor{LightCyan}{rgb}{0.88,1,1}

\newcolumntype{g}{>{\columncolor{Gray}}c}
\newcolumntype{y}{>{\columncolor{LightCyan}}c}
\newcolumntype{o}{>{\columncolor{pink}}c}

\newenvironment{remark}{\begin{remarkwr}\begin{upshape}}{\end{upshape}\end{remarkwr}}

\newcommand{\cA}{\mathcal A}
\newcommand{\cW}{\mathcal W}
\newcommand{\cT}{\mathcal T}
\newcommand{\cO}{{\mathcal O}}
\newcommand{\cG}{{\mathcal G}}
\newcommand{\cD}{{\mathcal D}}
\newcommand{\cL}{{\mathcal L}}

\newcommand{\sV}{\mathscr V}
\newcommand{\sS}{\mathscr S}
\newcommand{\sK}{\mathscr K}
\newcommand{\sD}{{\mathscr D}}
\newcommand{\sQ}{{\mathscr Q}}
\newcommand{\sH}{\mathscr H}
\newcommand{\sB}{{\mathscr B}}
\newcommand{\sR}{{\mathscr R}}

\newcommand{\Z}{\mathbb{Z}}
\newcommand{\Q}{\mathbb{Q}}
\newcommand{\R}{\mathbb{R}}
\newcommand{\C}{\mathbb{C}}
\newcommand{\A}{\mathbb{A}}
\newcommand{\F}{\mathbb{F}}
\newcommand{\T}{\mathbb{T}}
\newcommand{\PP}{\mathbb{P}}

\newcommand{\GL}{\mathrm{GL}}

\newcommand{\SL}{\mathrm{SL}}

\newcommand{\Gal}{\mathrm{Gal\,}}
\newcommand{\Aut}{\mathrm{Aut}}
\newcommand{\Hom}{\mathrm{Hom}}

\newcommand{\Div}{\mathrm{Div}}
\newcommand{\Pic}{\mathrm{Pic}}
\newcommand{\AJ}{\mathrm{AJ}}
\newcommand{\ord}{\mathrm{ord}}
\newcommand{\red}{\mathrm{red}}

\newcommand{\defeq}{\vcentcolon=}

\newcommand{\into}{\hookrightarrow}

\newcommand{\too}{\longrightarrow}								
\newcommand{\mapstoo}{\longmapsto}
\newcommand{\intoo}{\lhook\joinrel\longrightarrow}

\DeclareRobustCommand\ontoo{\relbar\joinrel\twoheadrightarrow}

\include{thebibliography}

\begin{document}

\title[The Gross--Kohnen--Zagier theorem via $p$-adic uniformization]
{The Gross--Kohnen--Zagier theorem  \\
via $p$-adic uniformization }

\author{Lea Beneish}
\author{Henri Darmon}
\author{Lennart Gehrmann}
\author{Mart\'{\i} Roset}

\address{L.~B.: University of North Texas, Denton, USA}
\email{lea.beneish@unt.edu}
\address{H.~D.: McGill University, Montreal, Canada}
\email{darmon@math.mcgill.ca}
\address{L.~G.: Bielefeld University, Bielefeld, Germany}
\email{gehrmann.math@gmail.com}
\address{M.~R.: McGill University, Montreal, Canada}
\email{marti.rosetjulia@mail.mcgill.ca}

\subjclass[2020]{11G18, 11F27, 11F30, 11F37}

\begin{abstract}
This article gives a new proof of the Gross--Kohnen--Zagier theorem for Shimura curves which exploits the $p$-adic uniformization of Cerednik--Drinfeld. The explicit description of CM points via this uniformization leads to an expression relating the Gross--Kohnen--Zagier generating series to the ordinary projection of the first derivative, with respect to a weight variable, of a $p$-adic family of positive definite ternary theta series.
\end{abstract}

\maketitle

\tableofcontents

\section{Introduction}
Let $S$ be a finite set of places of $\Q$ of odd cardinality containing $\infty$ and let $N^+$ be a square-free positive integer which is not divisible by any finite place in $S$. 
This datum gives rise to a {\em modular} or {\em Shimura curve} $X$ defined over $\Q$, which is an instance of an orthogonal Shimura variety.
Its set $X(\C)$ of complex points can be described in terms of an Eichler $\Z$-order $\sR$ of  level $N^+$ 
in a quaternion algebra $\sB$  over $\Q$  ramified exactly at $S - \{\infty\}$.
Namely, the set  $\sV$ of trace zero elements in $\sB$ equipped with the quadratic form $\sQ$ induced by the reduced norm is a quadratic space of signature $(1,2)$, and is anisotropic at all the places $v\in S-\{\infty\}$.
The action of $\sB^\times$ on $\sV$ via conjugation identifies $\sB^\times$ with the group of spinor similitudes of $\sV$.
It naturally acts on the conic $C_\sV\subseteq \PP(\sV)$ whose rational points over a field $E$ of characteristic $0$ are given by 
\begin{equation}\label{def: conic}
C_\sV(E) = \left\{ \ell\in \PP(\sV_E)\ \middle|\ \sQ(\ell) = \{0\} \right\}.
\end{equation}
Here and from now on, if $M$ is an abelian group, and $A$ is a ring, write $M_A \defeq M \otimes_\Z A$.
The group $\varGamma$ of units of $\sR$ acts discretely on the symmetric space
\[
\sK = C_\sV(\C) - C_\sV(\R)
\] 
associated to the orthogonal group of $\sV$.
The set $X(\C)$ of complex points of $X$ is identified with the quotient $\varGamma \backslash \sK$.

Given a vector $v\in \sV$ for which $\sQ(v) > 0$,  let  $\Delta(v)\subset \sK$ 
 be the  two points in $\sK$ represented by a vector orthogonal to $v$. 
 Each positive integer $D$ in
 \[
{\cD}_S \defeq \left\{ D \in \Z_{> 0}  \ \middle| \  \exists v \in \sV \mbox{ such that } \sQ(v) = D \right\} 
\]
gives rise to a zero-cycle  on $X$  by setting
\begin{equation}
\label{eqn:def-Delta-infty}
\Delta(D) \defeq \sum_{\substack{v \in \varGamma\backslash \sR_0, \\  \sQ(v) = D}} \frac{1}{ \#\mathrm{Stab}_{\varGamma/\{\pm 1\}}(v)  }\Delta(v) \in \Div(X(\C))_\Q,
\end{equation}
 where $\sR_0\subseteq \sV$ is the $\Z$-lattice $\sR \cap \sV$.
The divisor $\Delta(D)$, which is supported on a finite set of CM points on $X$, is a simple instance of a {\em Heegner divisor} on this Shimura curve.
The Gross--Kohnen--Zagier theorem asserts that the classes of $\Delta(D)$ in the Jacobian of $X$ can be packaged into a modular generating series of weight $3/2$.
Namely, let $\cL$ be the tautological line bundle of isotropic vectors whose spans 
are points of $\sK$.
This bundle is $\sB^\times$-equivariant and, therefore, descends to a line bundle on $X(\C)$, which is identified with the cotangent bundle of $X$.
In particular, it has a model over $\Q$.
Denote by $[\Delta]$ (resp.~$[\cL^\vee]$) the class in $\Pic(X)(\Q)$ of a divisor $\Delta$ (resp.~of the dual $\cL^\vee$ of the line bundle $\cL$) on $X$.
Then, the formal generating series 
\begin{equation}
\label{eqn:gkz}
 G(q) \defeq [\cL^\vee] + \sum_{D\in \cD_{S}} [\Delta(D)] q^{ D} \in \Pic(X)(\Q)_\Q[[q]],
\end{equation}
is a modular form of weight $3/2$ and level $\Gamma_0(4N)$, where $N$ is the product of $N^+$ with all finite places in $S$. 
\begin{remark}\label{rmk: characterization J valued mod forms}
	Let $A$ be an abelian group and let $f \in A[[q]]$ be a formal $q$-series with coefficients in $A$. Then $f$ is called a modular form of weight $3/2$ and level $\Gamma_0(4N)$ if for every morphism $\varphi\colon A \to \C$ the generating series $\varphi(f) \in \C [[q]]$, obtained by applying $\varphi$ to each of the coefficients of $f$, is the $q$-expansion of a modular form of weight $3/2$ and level $\Gamma_0(4N)$.
\end{remark}
The Gross--Kohnen--Zagier theorem was first proved in \cite{gkz} in the case of modular curves (i.e., where $S=\{\infty\}$) by calculating the Arakelov intersection pairings of the divisors $\Delta(D)$ with a fixed CM divisor.
It was extended by Borcherds \cite{borcherds} to the  setting of orthogonal groups of real signature $(n,2)$, encompassing Shimura curves as a special case where the underlying quadratic space is of signature $(1,2)$, as a consequence of his theory of singular theta lifts.
The work of Yuan, Zhang, and Zhang \cite{yzz} proves Theorem \ref{thm:main} in much greater generality, for certain orthogonal groups over totally real fields.

The goal of this article is to describe a new proof of the Gross--Kohnen--Zagier theorem in the case where $S\ne \{\infty\}$, i.e., when $X$ is not a modular curve. 
To simplify the exposition we will also assume that $2 \nmid N$.
\begin{theorem}
\label{thm:main}
The generating series $G(q) \in \Pic(X)(\Q)_\Q[[q]]$ of \eqref{eqn:gkz} is a modular form of weight $3/2$ and level $\Gamma_0(4N)$.
\end{theorem}

Our approach to this theorem rests on the fact that, at a finite place $p\in S$, the curve $X(\C_p)$ admits a $p$-adic analytic uniformization.
More precisely, $X(\C_p)$ can be described as the quotient of the $p$-adic upper half-plane by the discrete action of the norm one elements of an Eichler $\Z[1/p]$-order $R$ of level $N^+$ in the (definite) quaternion algebra ramified exactly at $S-\{p\}$. Furthermore, the Heegner divisors $\Delta(D)$ can be described $p$-adically in terms of this uniformization.
This immediately gives an expression of the generating series of degrees
\[
\deg(G)(q) = \deg(\cL^\vee) + \sum_{D \in \cD_S} \deg(\Delta(D))q^D
\]
in terms of definite ternary theta series, recovering a well-known modularity result (see for example \cite[Chapter 2]{HZ} and \cite[Theorem I]{K}).
Thus, it is enough to prove modularity of the generating series $TG(q)$ for Hecke operators of degree $0$, for which $TG(q)$ takes values in the $\Q$-rational points of the Jacobian $J$ of $X$.
The existence of a basis of modular forms with rational coefficients then reduces the problem to proving modularity of the generating series
\[
\log_\omega(TG)(q)  \defeq \sum_{D\in \cD_{S}} \log_\omega([T\Delta(D)]) q^D   \in \Q_p[[q]]
\]
for every cotangent vector $\omega$ of $J_{\Q_p}$ with associated $p$-adic formal logarithm $\log_\omega\colon J(\Q_p)\to \Q_p$.
For appropriate Hecke operators $T$, the $p$-adic description of the divisors $T\Delta(D)$ leads to an expression of this series as the ordinary projection of an infinitesimal $p$-adic deformation of a positive definite ternary theta series attached to the data $(\omega, R, T)$. More precisely,   these data give rise to a $p$-adic family of weighted theta series $\Theta_k$ of weight $k + 3/2$, $k \in \Z_p^\times$, whose specialization at weight $3/2$ vanishes (see Section \ref{subsec: p adic family Thetak} for its definition).
It then follows that its derivative with respect to $k$ evaluated at $k = 0$, denoted $\Theta_0'$, is a $p$-adic cusp form of weight $3/2$.
Let $e_{\ord}$ be $p$-ordinary projector acting on this space.
By a classicality result, $e_{\ord}(\Theta_0')$ is a cusp form of weight $3/2$ and level $\Gamma_0(4N)$. 
Let $\mathrm{pr}_1$ be the projector on the space of cusp forms of weight $3/2$ and level $\Gamma_0(4N)$ to the eigenspace of the Hecke operator $U_{p^2}$ of eigenvalue $1$.
The main contribution of this article is the following formula.
\begin{theorem}\label{thm: log(GKZ) = pr e Theta_0'}
	We have 
	\[
	\log_\omega(TG) = \mathrm{pr}_{1}(e_{\ord} (\Theta_0' )).
	\]
\end{theorem}
To summarize, the fact that $TG$ is a modular form is a consequence of the modularity of definite theta series and classicality of ordinary $p$-adic modular forms of half-integral weight. 
\begin{remark}
As the proof of Theorem \ref{thm:main} is a purely $p$-adic analytic one, it seems likely that it carries over to more general settings, e.g., to Shimura curves over totally real fields which admit a $p$-adic uniformization.
The assumption that $N^+$ is square-free stems from using $p$-adic families of \emph{scalar-valued} half-integral modular forms, which seem only well-behaved in that case.
Generalizing to arbitrary level likely requires a theory of families of vector-valued modular forms, which so far has only been developed in a few instances (see \cite{LN}).  
\end{remark}
The strategy sketched above bypasses the global height pairings studied by Gross, Kohnen, and Zagier, or the singular theta lifts that arise in the approach of Borcherds. It can be envisaged as fitting into the broader framework of a ``$p$-adic Kudla program", in which $p$-adic families of modular forms play much the same role as analytic families of Eisenstein series in the Archimedean setting.
Insofar as the generating series $G$ are among the simplest instances of the modular generating series arising in the Kudla program, it is hoped that the $p$-adic techniques described here will be more widely applicable, shedding light on the connection between special cycles on orthogonal and unitary Shimura varieties, $p$-adic Borcherds-type lifts, and $p$-adic families of theta series.
A general framework is laid out in the article \cite{DGL}, which introduces the notion of rigid meromorphic cocycles for orthogonal groups.
In \textit{loc.cit.~}modularity statements for generating series of special divisors on arithmetic quotients on higher-dimensional $p$-adic symmetric spaces are formulated.
A crucial input in their proof is the injectivity of the first Chern class when the arithmetic quotient has dimension $3$ and higher.
Theorem \ref{thm:main} complements the main theorem of \cite{DGL} by extending it to the case of curves, where the kernel of the Chern class map needs to be considered.

The organization of the article is as follows.  Section \ref{section: Cerednik--Drinfeld theorem}  explains the $p$-adic uniformization of $X$ and states the Gross--Kohnen--Zagier theorem in terms of this uniformization. Theorem \ref{thm:main p-adic} below describes the main result,
which  is somewhat more general than Theorem \ref{thm:main}, since the divisors $\Delta(D)$ are replaced by linear combinations of Heegner points weighted by Schwartz--Bruhat functions.
Section \ref{section: degrees of Heegner divisors} gives a short proof of the modularity of $\deg(G)$.
Section \ref{section: p-adic Abel--Jacobi map} introduces the $p$-adic Abel--Jacobi map, which gives an explicit description of the Jacobian of a Mumford curve.
This description is used in  Section \ref{section: values of p adic theta functions} to construct certain functionals on the Jacobian, whose values at Heegner points are computed in Section \ref{section: AJ images of Heegner divisors}.
In Section \ref{section: p-adic deformations of theta series}, we define the $p$-adic family $\Theta_k$, prove a classicality result regarding ordinary $p$-adic cusp forms of half-integral weight and prove the main Theorem \ref{thm: log(GKZ) = pr e Theta_0'}, which implies the Gross--Kohnen--Zagier theorem. Finally,  Section \ref{section: numerical example} illustrates the construction of $\Theta_k$ by presenting a concrete example where $S = \{ 7, 13, \infty\}$ and $p = 7$. In this case,   the ordinary projection of $\Theta_0'$ is computed numerically modulo $p$.

\vspace{0.2in}
\noindent \textbf{Acknowledgements:}
This paper was initiated in the Fall of 2020 during a seminar on ``the $p$-adic Kudla program" which was attended by three of the authors and around 10 other participants of a thematic semester in number theory at the {\em Centre de Recherches Math\'ematiques} (CRM) in Montreal. The authors are grateful to the participants of the seminar, notably 
Patrick Allen, Mathilde Gerbelli-Gauthier, Adrian Iovita, Debanjana Kundu, Antonio Lei, Mike Lipnowski, Adam Logan, Marc-Hubert Nicole, 
Alice Pozzi,  Giovanni Rosso, and Salim
Tayou, and to the CRM for hosting an in-person activity during some of the most difficult months of the Covid-19 pandemic.
It is also a pleasure to thank Jan Vonk for several enriching conversations related to this work.

Lea Beneish was supported by a CRM-ISM Postdoctoral Fellowship, and Henri Darmon   by an NSERC Discovery grant. Lennart Gehrmann acknowledges support by the Maria Zambrano grant for the attraction of international talent and by Deutsche Forschungsgemeinschaft (DFG, German Research Foundation) via a research fellowship and via the grant SFB-TRR 358/1 2023 -- 491392403.  Mart\'i Roset received the support of a fellowship from la Caixa Foundation (ID 100010434), code   LCF/BQ/EU21/11890132.

\section{The Cerednik--Drinfeld theorem}\label{section: Cerednik--Drinfeld theorem}
This section recalls the theorem of Cerednik--Drinfeld,
which gives a rigid analytic uniformization of $X$ at a finite prime $p\in S$  that is fixed once and for all. Moreover, we describe Heegner divisors in terms of this uniformization, which leads to a reformulation of the Gross--Kohnen--Zagier theorem in this setting.

\subsection{$p$-adic uniformization of $X$} The rigid analytic uniformization of $X$ proceeds by replacing the place $\infty$ in the complex uniformization of the introduction by the prime $p\in S$. To describe it, we need to introduce some notation. Let $B$ be the quaternion algebra over $\Q$ ramified exactly at the places in $S-\{p\}$. Let $R$ be an Eichler $\Z[1/p]$-order in $B$ of level $N^+$ and denote by $\Gamma$ the group of reduced norm $1$ elements in $R$. Let $Q$ be the restriction of the reduced norm to the space 
\[
V = \left\{ b\in B \ \middle| \ \mathrm{Tr}(b) = 0 \right\}
\]
of elements of reduced trace zero in $B$. It endows $V$ with the structure of a quadratic space of rank 3 over $\Q$, which is of real signature $(3,0)$.
Denote by $\langle\cdot,\cdot \rangle$ the symmetric bilinear form attached to $Q$, that is, $\langle v,w\rangle:=Q(v+w)-Q(v)-Q(w)$.
As in the case of the quadratic space $\sV$, the action of $B^\times$ on $V$ via conjugation identifies $B^\times$ with the group of spinor similitudes of $V$.
The intersection $R_0 = R\cap V$ is an even $\Z[1/p]$-lattice in $V$.

A $p$-adic symmetric space  is associated to the orthogonal group of $V_{\Q_p}$ as follows.
Similarly to \eqref{def: conic} denote by $C_V\subseteq \PP(V)$ the conic over $\Q$ attached to $V$ whose rational points over a field $E$ of characteristic $0$ are given by 
\[
C_V(E) = \left\{ \ell \in \PP(V_E) \ \middle| \ Q(\ell) = \{0\} \right\}.
\]
This conic has no rational points, but can be identified with the projective line $\PP$ over $\Q_p$ as follows:
choose an isomorphism of $B_{\Q_p}$ with the matrix ring $\mathrm{M}_2(\Q_p)$.
The conic is then identified with the space of non-zero nilpotent $2\times 2$-matrices up to scaling.
Mapping such a matrix to its kernel yields the desired isomorphism.
The action of $B_{\Q_p}^\times$ is identified with the action of $\GL_2(\Q_p)$ on $\PP_{\Q_p}$ via M\"obius transformations.
\begin{definition}
	The Drinfeld $p$-adic upper half plane is the $\Q_p$-rigid analytic space $\sH_p$, whose $E$-rational points for any complete extension $E/\Q_p$ is the set 
	\[
	\sH_p(E) \defeq C_V(E) - C_V(\Q_p) \simeq \PP_1(E) - \PP_1(\Q_p).
	\]
\end{definition}

We briefly explain the rigid analytic structure on $\sH_p$ in terms of the reduction map to the Bruhat--Tits tree.
The Bruhat--Tits tree, denoted $\cT$, is the graph whose set of vertices is the set of {\em unimodular } $\Z_p$-lattices in $V_{\Q_p}$.
Two unimodular $\Z_p$-lattices $L_1$ and $L_2$ are joined by an edge if they are \emph{$p$-neighbours}, that is,
\[
[L_1:L_1\cap L_2]=[L_2:L_1\cap L_2]=p.
\]
A choice of a vertex $L$ gives a smooth $\Z_{(p)}$-integral structure $C_L$ to the conic $C_V$.
If $L'$ is adjacent to $L$, then the image of $L\cap L'$ in $L/pL$ is a $2$-dimensional non-regular subspace, hence contains a unique isotropic subspace $\ell_{L'}$.
Mapping the edge $(L,L')$ to $\ell_{L'}$ yields a bijection between the set of lattices adjacent to $L$ and $C_L(\F_p) \simeq \PP_1(\F_p)$.
It follows that $\cT$ is homogeneous of degree $p+1$.
The set of vertices and edges of $\cT$ are denoted by $\cT_0$ and $\cT_1$ respectively, and   $\cT$ shall be viewed
as a disjoint union $\cT = \cT_0\sqcup \cT_1$.
Identifying the quadratic space $V_{\Q_p}$ with the set of trace zero endomorphisms of $\Q_p^2$ endowed with the norm form gives the more familiar description of the tree in terms of similarity classes of $\Z_p$-lattices in $\Q_p^2$.
Indeed, the assignment $[\Lambda] \mapsto {\Hom}_0(\Lambda,\Lambda)$ is a bijection between such similarity classes and unimodular lattices in $V_{\Q_p}$. Moreover, two classes $[\Lambda_1]$, $[\Lambda_2]$ are joined by an edge if they admit representatives  $\Lambda_1$ and $\Lambda_2$ satisfying $p\Lambda_1 \subset \Lambda_2 \subset \Lambda_1$.
From this description one easily deduces that $\cT$ is indeed a tree.
The identification of the two graphs is compatible with the natural actions of $B_{\Q_p}^\times$ and $\GL_2(\Q_p)$.
We define a notion of parity on the vertices of $\cT$ by requiring that every edge connects an even vertex with an odd one.
There are exactly two possible choices for this and we choose one of them.
The action of the elements of reduced norm one in $B_{\Q_p}$ on $\cT$ is parity-preserving.

We proceed by describing the well-known \emph{reduction map} 
\[
\red\colon \sH_p(\C_p) \too \cT
\]
in the language of quadratic forms.
For that let $\cO_{\C_p}$ denote the ring of integers of $\C_p$ and $\mathfrak{m}\subseteq \cO_{\C_p}$ its maximal ideal.
Every unimodular $\Z_p$-lattice $L\subseteq V_{\Q_p}$ induces a reduction map
\[
C_V(\C_p)=C_L(\cO_{\C_p})\ontoo C_L(\overline{\F_p}).
\]
\begin{enumerate} 
\item 
Let $L\subseteq V_{\Q_p}$ be a unimodular $\Z_p$-lattice.
Then $\red^{-1}(L)$ is the complement of the $p+1$ residue discs around the points in $C_L(\F_p)$.
\item Let $L, L'\subseteq V_{\Q_p}$ be two unimodular $\Z_p$-lattices that are $p$-neighbours and $\ell_{L'}\in C_L(\F_p)$ the corresponding isotropic line.
The preimage of the edge $(L,L')$ under the reduction map consists of those elements $z\in C_L(\cO_{\C_p})$ that are congruent to $\ell_{L'}$ modulo $\mathfrak{m}$ but not modulo $p$.
\end{enumerate}
One readily checks that the reduction map is $B_{\Q_p}^\times$-equivariant.

A \emph{finite closed subgraph} of $\cT$ is a finite set $\cG\subset \cT$ 
satisfying
\[
(v_1,v_2)\in \cG\cap \cT_1 \Rightarrow v_1, v_2 \in \cG\cap \cT_0.
\]
A \emph{standard affinoid subset} of $\sH_p$ is a set of the form $\red^{-1}(\cG)$, where $\cG$ is a finite closed subgraph of $\cT$. 

\begin{definition} 
	A function on $\sH_p$ is said to be {\em rigid analytic} if its restriction to any standard affinoid subset $\cA\subset\sH_p$ can be written as a uniform limit of rational functions having poles outside of $\cA$. A function on $\sH_p$ is said to be {\em rigid meromorphic} if it is the quotient of two rigid analytic functions, where the denominator is non-zero.
\end{definition}

The group $\Gamma$ acts naturally on $\sH_p$ by conjugation. This action is discrete because $\Gamma$ is a $p$-arithmetic subgroup of an algebraic group that is compact at $\infty$. It follows from there that the quotient space $\Gamma \backslash \sH_p$ has a natural structure of a rigid analytic variety over $\Q_p$. On the other hand, the analytification of $X$  gives a rigid analytic space over $\Q_p$. The Cerednik-Drinfeld theorem states that these two spaces can be identified after base change to the unramified quadratic extension $\Q_{p^2}$ of $\Q_p$.
This identification depends on choices. To make this precise, let us introduce the following notation:
for a finite set $\Sigma$ of places of $\Q$ write ${\A^\Sigma} \subset \prod_{v\notin \Sigma} \Q_v$ for the ring of finite ad\'eles away from $\Sigma$. Moreover, let $\hat{\Z}$ (resp.~$\hat{\Z}^{(p)}$) be the maximal order of $\A^{\infty}$ (resp.~of $\A^{p,\infty}$). Given a finitely generated $\Z[1/p]$-module $M$, we put $\hat{M}=M\otimes \hat{\Z}^{(p)}$.
Fix an identification
\begin{equation}\label{eq: identification sV with V away from p}
	\sV_{\A^{p,\infty}} \simeq V_{\A^{p,\infty}}
\end{equation}
sending the $\hat{\Z}^{(p)}$-lattice $\hat{\sR_0}$ to $\hat{R_0}$.

\begin{theorem}[Cerednik--Drinfeld]\label{thm: Cerednik--Drinfeld uniformization}
	The identification \eqref{eq: identification sV with V away from p} induces an isomorphism 
	\begin{equation}
	\label{eqn:cerednik-drinfeld}
	X \xlongrightarrow{\sim} \Gamma \backslash \sH_p.
	\end{equation}
	of rigid analytic spaces over $\Q_{p^2}$.
\end{theorem}
\begin{proof}
	See \cite{C}, \cite{D} and \cite{BC}.
\end{proof}
   
\subsection{$p$-adic analytic description of Heegner divisors}
In analogy with the cycles defined in the introduction, every non-zero element $v\in V$ yields a cycle $\Delta(v)$ on $\sH_p$: $\Delta(v)$ is the sum of those points in $\sH_p$ that are orthogonal to $v$.
This cycle has degree $0$ or $2$ depending on whether the orthogonal complement of $v$ in $V_{\Q_p}$ represents $0$ or not.
In other words, $\Delta(v)\neq 0$ if and only if $\sqrt{-Q(v)}\not \in \Q_p$.
By the Hasse--Minkowski theorem the set $\cD_S$ from the introduction is characterized locally.
In particular, one gets the description
\[
{\cD}_S = \left\{ D \in \Z \ \middle| \ \exists v \in V - \{0\} \mbox{ such that } Q(v) = D \mbox{ and } \Delta(v)\neq 0 \right\},
\]
where we used that $B$ ramifies exactly at $S - \{ p\}$ while  $\sB$ ramifies exactly at $S - \{ \infty\}$.
\begin{lemma}\label{lem: howmanylattices}
Let $D$ be an element of $\cD_S$ and $v\in V$ with $Q(v)=D$.
\begin{enumerate}
\item If $\ord_p(D)=0$, there exists a unique unimodular $\Z_p$-lattice $L$ in $V_{\Q_p}$ containing $v$.
The support of $\Delta(v)$ is contained in $\red^{-1}(L)$.
\item If $\ord_p(D)=1$, there exist exactly two unimodular $\Z_p$-lattices $L_1,L_2$ in $V_{\Q_p}$ containing $v$, which are $p$-neighbours.
The support of $\Delta(v)$ is contained in $\red^{-1}((L_1,L_2))$.
\end{enumerate}
\end{lemma}
\begin{proof}
Let $W$ be the orthogonal complement of $v$ in $V_{\Q_p}$.
As $W$ is anisotropic, it contains a unique maximal $\Z_p$-lattice $L_W$, on which $Q$ takes values in $\Z_p$ (see for example \cite[Lemma 11]{AbramenkoNebe}). Moreover, it is characterized by the property that its discriminant module is an $\F_p$-vector space.
Suppose there exists a unimodular lattice $L$ containing $v$.
By \cite[Lemma 1.1]{DGL} the discriminant module of $L\cap W$ is an $\F_p$-vector space and thus equal to $L_W$.
In particular, $L$ must contain $\Z_p v\oplus L_W$.

If $\ord_{p}(D)=0$, one easily checks that $L_W$ is unimodular. Thus, $\Z_p v\oplus L_W$ is the unique unimodular $\Z_p$-lattice containing $v$.
If $\ord_p(D)=1$, then $L_W$ is of index $p$ in its dual. Thus, the discriminant module $M$ of $\Z_p v \oplus L_W$ is a $2$-dimensional $\F_p$-vector space. A quick calculation shows that $Q$ induces a hyperbolic form on $M$.
There exist exactly two self-dual lattices containing $\Z_p v\oplus L_W$ corresponding to the two isotropic lines in $M$.

Let $\sigma_v$ be the unique simplex of $\cT$ corresponding to $v$ and $\Q(v)$ the $\Q$-subalgebra of $B$ generated by $v$.
Since $\Q(v)^\times$ fixes $v$, it follows that it also fixes $\sigma_v$.
The statements about the support of $\Delta(v)$ follow from the $B_{\Q_p}^\times$-invariance of the reduction map.
\end{proof}

The space $\sS(V_{\A^{p,\infty}})$ of $\Z$-valued Schwartz--Bruhat functions on $V_{\A^{p,\infty}}$ admits an action of $B_{\A^{p,\infty}}^\times$ induced by the conjugation action on $V_{\A^{p,\infty}}$. 
Attached to an $\hat{R}^\times$-invariant function $\Phi \in \sS(V_{\A^{p,\infty}})$ and a non-zero rational number $D$ is the zero-cycle
\begin{equation}
\label{eqn:def-Delta-p}
\Delta_{\Phi, \Gamma}(D) = \sum_{\substack{v\in \Gamma \backslash V, \  Q(v) = D}} \frac{1}{\# \mathrm{Stab}_{\Gamma/ \{\pm 1 \}}(v) }\Phi(v) \Delta(v)   \in   \Div(\Gamma \backslash \sH_p)_{\Q},
\end{equation}
on $\Gamma \backslash \sH_p$.
Note the formal similarities between \eqref{eqn:def-Delta-infty} and \eqref{eqn:def-Delta-p} when $\Phi$ is the characteristic function of $\hat{R}_0$, that  will simply be denoted as $1_{R_0}$. We proceed to make them precise. By the theory of complex multiplication, the Heegner points appearing in the divisors $\Delta(D)$ of the introduction are defined over $\overline{\Q}$. Hence, after fixing an embedding $\overline{\Q} \subset \C_p$,  these divisors  can be viewed as elements in $\Div(X(\C_p))$.

\begin{proposition}\label{prop: relation complex and p-adic description of Heegner points}
Let $D \in \cD_S$. The cycle $\Delta_{ 1_{R_0}, \Gamma}(D)$ of \eqref{eqn:def-Delta-p} viewed as an element of $\Div(X(\C_p))$ via the Cerednik--Drinfeld uniformization theorem is equal to the cycle $\Delta(D)$ of \eqref{eqn:def-Delta-infty}.
\end{proposition}
\begin{proof}
	Since $R_0$ is a $\Z[1/p]$-lattice, 
	$$\Delta_{1_{R_0}, \Gamma}(D) = \Delta_{1_{R_0}, \Gamma}(Dp^{2n})$$
	 for every $n \geq 0$. On the other hand, the fact that $\sB$ is ramified at $p$ implies that
	\[
	\ord_p(\sQ(\sR_0 - p\sR_0))\subseteq \{0,1\}.
	\]
Therefore, multiplication by $p$ gives a bijection between elements of length $D$ and elements of length $Dp^2$ in $\sR_0$ for every $n\geq 0$, which yields the equality $\Delta(D) = \Delta(Dp^{2n})$.
It is then enough to prove the identification when $D \in \cD_S$ is such that $\ord_p(D) \in \{ 0, 1\}$. The case when $\ord_p(D) = 0$ is treated in Theorem 5.3 of \cite{BD} and the case when $\ord_p(D) = 1$ follows from \cite[Section 3.3]{LP} and \cite[Proposition 5.12]{M}. 
\end{proof}

The Weil representation attached to $V$ and the standard character $\psi\colon \Q \backslash \A \to \C^\times$ induces an action of the metaplectic group $\widetilde{\SL_2}(\A^{p,\infty})$ on $\sS(V_{\A^{p,\infty}})$ that commutes with the $B_{\A^{p,\infty}}^\times$-action.
For $M \geq 0$, let $K_0(M)^{(p)}$ be the subgroup of $\SL_2(\A^{p,\infty}) $ consisting of matrices in $\SL_2(\hat{\Z}^{(p)})$ with left lower entry divisible by $4M$.
Since $K_0(4M)^{(p)}$ splits the exact sequence
\[
1 \too \{\pm 1 \} \too \widetilde{\SL_2}(\A^{p,\infty}) \too \SL_2(\A^{p,\infty}) \too 1
\]
defining the metaplectic group, it can be regarded as a subgroup of $\widetilde{\SL_2}(\A^{p,\infty})$. We similarly define $K_0(4M)$ and view it as a subgroup of $\SL_2(\A^{\infty})$ and of $\widetilde{\SL_2}(\A^{\infty})$.
\begin{definition}
A Schwartz--Bruhat function $\Phi\in\sS(V_{\A^{p,\infty}})$ is called special if
\begin{enumerate}
\item\label{spec1} $\Phi$ is $\hat{R}^\times$-invariant,
\item\label{spec2} $\Phi$ is $K_0(4N)^{(p)}$-invariant, and
\item\label{spec3} $\Phi(pv)=\Phi(v)$ for all $v\in V$.
\end{enumerate}
\end{definition}
The characteristic function $1_{R_0}$ is the prime example of a special Schwartz--Bruhat function.
Let $\Phi$ be a special Schwartz--Bruhat function.
Property \eqref{spec2} implies that for $D\in \Z_{(p)}- \{0\}$ we have $\Delta_{\Phi, \Gamma}(D)=0$ unless $D\in\cD_S$. 
Furthermore, the equality
\begin{equation}\label{eq: special}
\Delta_{\Phi, \Gamma}(p^2D)=\Delta_{\Phi, \Gamma}(D)
\end{equation}
holds for all $D\in\cD_S$ by Property \eqref{spec3}.
The remainder of this work will solely be concerned in proving the following $p$-adic analytic version of the Gross--Kohnen--Zagier theorem, which implies Theorem \ref{thm:main} in view of the previous proposition and the fact that $1_{R_0}$ is special.
\begin{theorem}\label{thm:main p-adic}
	Let $\Phi$ be a special Schwartz--Bruhat function.
	The generating series 
	\[
	G_{\Phi, \Gamma}(q) \defeq  \Phi(0)[\cL^\vee] + \sum_{D \in \cD_S} [\Delta_{\Phi, \Gamma}(D)] q^D \in  \Pic(\Gamma \backslash \sH_p)_\Q [[q]],
	\]
	is a modular form of weight $3/2$ and level $\Gamma_0(4N)$.
\end{theorem} 
\begin{remark}
	The divisors described above are compatible under pullback in the following sense.
	Let $\Phi$ be a Schwartz--Bruhat function on $V_{\A^{p,\infty}}$ invariant under $\hat{R}^\times$. Suppose that $R'$ is an Eichler $\Z[1/p]$-order contained in $R$, denote by $\Gamma'$ the group of reduced norm $1$ units in $R'$ and consider $\pi\colon \Gamma' \backslash \sH_p \to \Gamma \backslash \sH_p$. Then it can be seen in a similar way as in the proof of \cite[Proposition 5.10]{Ku1} that, for every $D \in \cD_S$,
	\[
	\pi^\ast(\Delta_{\Phi, \Gamma}(D)) = \Delta_{\Phi, \Gamma'}(D).
	\] 
	Using that $\pi_\ast \circ \pi^\ast$ is equal to multiplication by the degree of $\pi$ on $\mathrm{Div}(\Gamma \backslash \sH_p)$ and the previous identity, we deduce that $R$ can be replaced by $R'$ 
in the proof of  Theorem \ref{thm:main p-adic}.
In particular, we will assume from now on that $\bar{\Gamma}:=\Gamma / \{\pm 1\}$ is torsion-free by choosing an appropriate level $N^+$.
This will simplify some calculations as the group $\bar{\Gamma}$ will act freely on $\mathcal H_p$ and $\mathcal T$. 
In particular, it is a free group on finitely many generators.
Moreover, under this assumption the coefficients of the divisors $\Delta_{\Phi, \Gamma}(D)$ are integral.
\end{remark}

\subsection{Hecke action on divisors}
Let $\T^N$ be the integral Hecke algebra away from $N$, which is generated by the standard generators $\{T_\ell\}_{\ell \nmid N}$ (see \cite[Section 1.2]{JL} for its definition). We conclude the section describing the action of $\T^N$ on the divisors as well as on the space of $\hat{R}^\times$-invariant Schwartz--Bruhat functions. 
Let $\ell$ be a prime not dividing $N$ and fix $\alpha \in B^\times \cap R$ an element of reduced norm $\ell$.
Consider the maps
\[
\Gamma \backslash \sH_p \xleftarrow{\pi_1} (\alpha^{-1} \Gamma \alpha \cap \Gamma) \backslash \sH_p \xlongrightarrow{\alpha} (\Gamma \cap \alpha \Gamma \alpha^{-1}) \backslash \sH_p \xlongrightarrow{\pi_2} \Gamma \backslash \sH_p. 
\]
Then, define the action of the Hecke operator $T_\ell$ on divisors as
\[
T_\ell (\Delta) \defeq (\pi_{2,\ast} \circ \alpha \circ \pi_1^\ast) (\Delta)\quad \mbox{for }\Delta\in \Div(\Gamma\backslash \sH_p).
\]
On the other hand, the action of on $\hat{R}^\times$-invariant Schwartz--Bruhat functions is determined as follows. If $\Gamma = \sqcup_{j}(\Gamma \cap \alpha \Gamma \alpha^{-1})\delta_j$ for $\{ \delta_j\}_j \subset \Gamma$ we define
\[
T_\ell (\Phi) \defeq  \sum_{j} \Phi\cdot (\alpha^{-1}\delta_j),
\]
where if $\beta \in B^\times$, $\Phi \cdot \beta(v) \defeq \Phi(\beta v \beta^{-1})$. Note that, since $\ell \nmid N$, $R \cap \alpha R \alpha^{-1}$ is an Eichler $\Z[1/p]$-order. Hence, by strong approximation, the double coset space $\left( \hat{R} \cap \alpha \hat{R} \alpha^{-1} \right)^\times \backslash \hat{B}^\times   / B^\times$ has precisely one element. Using this, together with the fact that $ R  \cap \alpha R \alpha^{-1}$ has an element of reduced norm $p$ (\cite[Lemma 1.5]{BD1}), we deduce: for the same $\{\delta_j\}_{j} \subset \Gamma$ as above $\hat{R}^\times = \sqcup_j (\hat{R} \cap \alpha \hat{R}^\times \alpha^{-1}) \delta_j$. Hence, $\hat{R}^\times \alpha^{-1}\hat{R}^\times = \sqcup_j \hat{R}^\times \alpha^{-1}\delta_j$ and it follows from there that $T_\ell (\Phi)$ is $\hat{R}^\times$-invariant.
It follows from this description that the Hecke action preserves the subspace of special Schwartz--Bruhat functions.
\begin{lemma}\label{lemma: Hecke action divisors and Schwartz--Bruhat functions}
	Let $\Phi$ be a Schwartz--Bruhat function on $V_{\A^{p,\infty}}$ invariant under $\hat{R}^\times$.
	The following identity of divisors holds:
	\[
	T_\ell(\Delta_{\Phi, \Gamma}(D)) = \Delta_{T_\ell (\Phi), \Gamma}(D).
	\]	
\end{lemma}
\begin{proof}
	Using that $\bar{\Gamma}$ is torsion-free, the proof can be completed by following the next steps, which are proved in a similar way than \cite[Proposition 5.9 and Proposition 5.10]{Ku1}:
	\begin{equation*}
		\begin{split}
			T_\ell (\Delta_{\Phi, \Gamma}(D))  & = (\pi_{2,\ast} \circ \alpha \circ \pi_1^\ast) \left( \Delta_{\Phi, \Gamma}(D) \right) \\ & = (\pi_{2,\ast} \circ \alpha) \left(  \Delta_{\Phi, \alpha^{-1}\Gamma\alpha \cap \Gamma}(D)  \right) \\ & = \pi_{2,\ast} \circ \alpha \Delta_{\Phi \cdot \alpha^{-1}, \Gamma \cap \alpha\Gamma\alpha^{-1}}(D)  \\ & = \Delta_{T_\ell(\Phi), \Gamma}(D).
		\end{split}
	\end{equation*}
\end{proof}
\begin{remark}
	When it is clear from the context that we are viewing $\Phi$ as a $\Gamma$-invariant Schwartz--Bruhat function, we will write $\Delta_\Phi(D)$ (resp.~$G_\Phi$) to denote $\Delta_{\Phi, \Gamma}(D)$ (resp. $G_{\Phi, \Gamma}$).
\end{remark}

\section{Modularity of degrees of Heegner divisors}\label{section: degrees of Heegner divisors}

Fix a special Schwartz--Bruhat function $\Phi$.
In this section, we prove that 
\[
\deg(G_\Phi)(q) = \Phi(0)\deg(\cL^\vee) + \sum_{D\in \cD_S} \deg(\Delta_\Phi(D))q^D
\]
is a modular form by comparing $\deg(G_\Phi)$ to a genus theta series attached to $V$.

Fix $L_0, \dots ,L_r$ unimodular $\Z_p$-lattices in $V_{\Q_p}$ that give a set of representatives of $\Gamma \backslash \cT_0$.
For every $i$ consider the ternary theta series attached to the Schwartz--Bruhat function $\Phi \otimes 1_{L_i}$ on $V_{\A^\infty}$
\[
\Theta_i = \sum_{v \in V} \Phi(v)1_{L_i}(v)q^{Q(v)}.
\]
Note that theta series $\Theta_i$ only depends on the class of $L_i$ in $R^\times \backslash \cT_0$. Since $\Phi$ is invariant under $K_0(4N/p)^{(p)}$ and $L_i$ is a unimodular $\Z_p$-lattice, $\Phi \otimes 1_{L_i}$ is invariant under $K_0(4N/p)$. It is well known that $\Theta_i$ is a modular form of weight $3/2$ and level $\Gamma_0(4N/p)$ for every $i$ (see for example Theorem 4.1 of \cite{BoAutomorphic}).
Define the modular form
\[E_\Phi \defeq \sum_{i = 1}^r \Theta_{i}.\]
The following lemma relates the degrees of those $\Delta_\Phi(D)$ with $\ord_{p}(D)\in\{0,1\}$ with the corresponding Fourier coefficients of $E_\Phi$.
\begin{lemma}\label{prop: relation aD(E) and degrees}
	Let $D \in \cD_S$, then:
	\begin{enumerate}
		\item if $\ord_p(D)=0$, then $2a_D(E_\Phi) = \deg(\Delta_\Phi(D))$, and
		\item if $\ord_p(D) = 1$ then, $a_{D}(E_\Phi) = \deg(\Delta_\Phi(D))$.
	\end{enumerate}
\end{lemma}
\begin{proof}
Since $\bar{\Gamma}$ is torsion-free, it does not stabilize any vertex of $\cT_0$.
	Thus, Lemma \ref{lem: howmanylattices} implies that
	\begin{align*}
	\bigsqcup_{i=1}^{r} \left\{ v \in V \cap L_i \ \middle| \  Q(v) = D \right\} &\xlongrightarrow{\sim} \left\{ v \in \Gamma \backslash V \ \middle| \  Q(v) = D  \right\}\\
	v &\mapstoo [v]
	\end{align*}
	is bijective if $\ord_p(D) = 0$ and surjective and two-to-one if $\ord_p(D) = 1$, which implies the assertion.
\end{proof}

Let $M \in \Z_{>0}$ and $k$ such that $2k \in \Z_{>0}$. When $k$ is a half-integer, we will always assume that the level $M$ is divisible by $4$.
Denote by $M_{k}(\Gamma_0(M))$ (resp.~$S_{k}(\Gamma_0(M))$) the space of modular forms (resp.~cusp forms) forms of weight $k$ and level $\Gamma_0(M)$.
Consider the subspaces $M_{k}(\Gamma_0(M), \Z)$ (resp.~$S_{k}(\Gamma_0(M), \Z)$) of forms whose $q$-expansion has integral coefficients and for any abelian group put
\[
M_{k}(\Gamma_0(M), A) \defeq M_k(\Gamma_0(M), \Z) \otimes_\Z A\ \  \mbox{and}\ \ S_{k}(\Gamma_0(M), A) \defeq S_k(\Gamma_0(M), \Z) \otimes_\Z A.
\]
We view these as subspaces of the group $A[[q]]$ of formal $q$-series with coefficients in $A$.
By \cite[Lemma 8]{SerreStark} there exists a basis of $M_{k}(\Gamma_0(M))$ consisting of forms with integral coefficients.
Hence, the natural homomorphisms
\begin{align*}
M_{k}(\Gamma_0(M), \C)\xlongrightarrow{\sim}M_{k}(\Gamma_0(M))\ \ \mbox{and}\ \ S_{k}(\Gamma_0(M), \C)\xlongrightarrow{\sim}S_{k}(\Gamma_0(M))
\end{align*}
are bijective. 
We now introduce several operators acting on $M_{3/2}(\Gamma_0(M), A)$.
For that let
\[
f = \sum_{n\geq 0} a_n q^n \in A[[q]]
\]
be a formal $q$-series with coefficients in $A$.
Define
\begin{align*}
T_{p^2}(f) &\defeq \sum_{n \geq 0} \left( a_{p^2n} + \left( \frac{-n}{p}\right) a_n + p a_{n/p^2}\right)q^n,\\
U_{p^2}(f)  &\defeq \sum_{n \geq 0} a_{p^2n}q^n,
\end{align*}
and put $V_{p^2}f \defeq T_{p^2}f  - U_{p^2}f$.
Now suppose that $f\in M_{3/2}(\Gamma_0(M), R)$ is a modular form.
Then $T_{p^2}f\in M_{3/2}(\Gamma_0(M), R)$ if $p\nmid M$ and, in case $p\mid M$, we have $U_{p^2}f\in M_{3/2}(\Gamma_0(M), R)$.

\begin{proposition}\label{prop: Tp^2 eigenvalue of E}
	The modular form $E_\Phi$ is an Eisenstein series of level $\Gamma_0(4N/p)$.
	In particular, it satisfies $T_{p^2}(E_\Phi) = (p+1)E_\Phi$.
\end{proposition}
\begin{proof}
	Observe that, if we denote by $R^\times_+$ the subgroup of $R^\times$ consisting of units whose reduced norm has even $p$-adic valuation, we have $[R^\times:R^\times_+] = 2$. It follows from there that $r$ is even and that we can suppose that the representatives $\{L_i\}_i$ of $\Gamma \backslash \mathcal T_0$ are arranged so that $\{ L_1, \dots , L_{r/2}\}$ give a set of representatives of $R^\times \backslash \mathcal T_0$. Then, since $\Theta_{L_i}$ only depends on the class of $L_i$ in $R^\times \backslash \mathcal T_0$, we have $E_\Phi = 2 \sum_{i= 1}^{r/2} \Theta_{L_i}$.
Since all Eichler $\Z[1/p]$-orders of level $N^+$ in $B$ are conjugate, the set $\{ L_1, \dots , L_{r/2}\}$ forms a genus of integral quadratic forms.
The Siegel--Weil formula (see \cite[Theorem 4.1 (ii)]{K}) states that the sum of theta functions in a given genus - and therefore $E_\Phi$ - is an Eisenstein series.
\end{proof}
Since $\Phi$ is special, \eqref{eq: special} implies that
\[
U_{p^2}(G_\Phi(q))=G_{\Phi}(q).
\]
We proceed to modify $E_\Phi$ so that it becomes invariant under $U_{p^2}$ as well.
For that, put $E_{\Phi}^{1} \defeq E - V_{p^2}(E) \in M_{3/2}(\Gamma_0(4N))$.
\begin{corollary}\label{lem: Eisensteinstabilization}
	We have $U_{p^2}(E_{\Phi}^{1})  = E_{\Phi}^{1}$.
\end{corollary}
\begin{proof}
	Since $(U_{p^2}\circ V_{p^2}) (E) = pE$ (which can be verified directly from the description of $U_{p^2}$ and $V_{p^2}$ given above), we have
	\[
	U_{p^2} (E_{\Phi}^{1}) = U_{p^2}(E) - pE.
	\]
	Using that $U_{p^2} = T_{p^2} - V_{p^2}$ and Proposition \ref{prop: Tp^2 eigenvalue of E} yields the desired result.
\end{proof}
We can finally prove the main result of this section. 
\begin{proposition}\label{prop: degGPhi is modular}
	The equality $\deg(G_\Phi)(q) = E_{\Phi}^{1}$ holds.
	In particular, $\deg(G_\Phi)(q)$ is an Eisenstein series of weight $3/2$ and level $\Gamma_0(4N)$.
\end{proposition}
\begin{proof}
	By Lemma \ref{lem: Eisensteinstabilization}, the equality $U_{p^2}(E_{\Phi}^{1}) = E_{\Phi}^{1}$ holds.
	On the other hand, since $\Phi$ is special we have $\deg(\Delta_\Phi(Dp^2)) = \deg(\Delta_\Phi(D))$ for all $D$.
	Hence, it is enough to verify that the Fourier coefficients of $E_{\Phi}^{1}$ and of $\deg(G_\Phi)$ are equal in the following cases:
	\begin{itemize}
	\item If $\ord_p(D) = 1$, the second point of Lemma \ref{prop: relation aD(E) and degrees} implies
		\[
		a_{D}(E_{\Phi}^{1}) = a_{D}(E_\Phi) = \deg(\Delta_\Phi(D)).
		\]
		
		\item If $\ord_p(D) = 0$ and $\left(\frac{-D}{p} \right) = -1$, the first point of Lemma \ref{prop: relation aD(E) and degrees} gives
		\[
		a_D(E_{\Phi}^{1}) = 2a_D(E_\Phi) = \deg(\Delta_\Phi(D)).
		\]

		\item If $\ord_p(D) = 0$ and $\left(\frac{-D}{p} \right) = 1$, one calculates
		\[
		a_D(E_{\Phi}^{1}) = a_D(E_\Phi) - a_D(E_\Phi) = 0.
		\]
		On the other hand, we have that $\Delta_\Phi(D) = 0$, as $G_\Phi(q)$ is supported only on non-negative integers that belong to $\cD_S$.
		
		\item If $D = 0$, we have $a_0(E_{\Phi}^{1}) = \Phi(0) (1-p) r$, where we recall that $r = \#(\Gamma \backslash \mathcal T_0)$. Now, since $\bar{\Gamma}$ is torsion-free, it follows that $\Gamma \backslash \mathcal T$ is a $(p+1)$-regular graph.
		Thus, we readily compute its first Betti number
		\[
		g(\Gamma \backslash \mathcal T) = 1- \# (\Gamma \backslash \mathcal T_0) + \#(\Gamma \backslash \mathcal T_1) = 1 - r + \frac{p+1}{2} r,
		\]
		which by \cite[Theorem 5.4.1]{FvdP} equals the genus $g$ of $X$.
		The degree of the cotangent bundle of $X$ is equal to $2g-2$.
		This implies that 
		\[
		a_0(E_{\Phi}^{1}) = \Phi(0)(1-p) r = \Phi(0)(2 - 2g) = \Phi(0)\mathrm{deg}(\mathcal L^\vee).
		\]
	\end{itemize}
	Therefore, we obtain the desired equality $\deg(G_\Phi) = E_{\Phi}^{1}$.
\end{proof}

\section{The Abel--Jacobi map}\label{section: p-adic Abel--Jacobi map}

Because the curve $X$ is a {\em Mumford curve} over $\Q_{p^2}$, its Jacobian, denoted by $J$, has purely toric reduction and admits a concrete description in terms of equivalence classes of automorphy factors of rigid meromorphic functions in $\sH_p$. In this section, we explain how the class in $J$ of a degree zero divisor can be described explicitly in these terms. Then, we introduce the notion of divisors of strong degree $0$, for which there exists a preferred choice of automorphy factor describing its class in $J$. Finally, we use this notion to reduce Theorem \ref{thm:main p-adic} to the case where all divisors appearing as coefficients of the generating series $G_\Phi$ have strong degree $0$.

\subsection{Definition and properties of the Abel--Jacobi map} A {\em formal divisor} on $\sH_p$ is a formal, possibly infinite $\Z$-linear combination of points in $\sH_p$. A formal divisor
\[
\hat\sD = \sum_{x\in \sH_p} m_x (x)
\]
is said to be {\em discrete} if the formal divisor
\[
\hat \sD\cap\cA \defeq \sum_{x\in \cA} m_x (x)
\]
is an actual divisor, i.e., involves a finite sum for all standard affinoid subsets $\cA\subset \sH_p$. The set of all discrete formal divisors on $\sH_p$ is denoted by ${\Div}^\dag(\sH_p)$.
Denote by $\Div(\sH_p)$ (resp.~$\Div^0(\sH_p))$ the subset of finite divisors (resp.~finite divisors of degree $0$).
 The quotient map $\pi\colon\sH_p \to \Gamma\backslash \sH_p$ induces pullback and pushforward maps
\[ 
\pi_\ast\colon {\Div}(\sH_p) \too {\Div}(\Gamma\backslash\sH_p),
\qquad \pi^\ast\colon {\Div}(\Gamma\backslash \sH_p) \too {\Div}^\dag(\sH_p),
\]
since $\Gamma$ acts on $\sH_p$ with discrete orbits. Given $\Delta\in {\Div}(X(\C_p))$, let $ \sD\in {\Div}(\sH_p)$ 
and $\hat\sD\in {\Div}^\dag(\sH_p)$ be (formal) divisors satisfying
\begin{equation}
\label{eqn:defsdsdhat}
 \pi_\ast(\sD) = \Delta, \quad  \hat \sD = \pi^\ast(\Delta).
 \end{equation}
The divisor $ \sD$ is not unique, while the formal divisor  $\hat\sD$ is completely determined by $\Delta$.

Given any degree zero divisor $\sD$ on $C_V(\C_p) \simeq \PP_1(\C_p)$, there is a rational function $f_\sD$ on $C_V(\C_p)$ having $\sD$ as a divisor, which is unique up to a multiplicative constant. A  rational function $f$ is extended multiplicatively to any divisor $\sD = \sum_{x \in C_V(\C_p)} m_x \cdot (x)$ by setting
\[
f(\sD) \defeq \prod_{x \in C_V(\C_p)} f(x)^{m_x}.
\]
\begin{definition}
	The {\em Weil symbol} attached to two degree zero divisors $\sD_0$ and $\sD_1$ on $C_V(\C_p)$ with disjoint supports is the quantity
	\[
	[\sD_0;\sD_1] \defeq f_{\sD_0}(\sD_1) \in \C_p.
	\]	
\end{definition}
The Weil symbol generalises the familiar cross-ratio which one recovers when $\sD_0$ and $\sD_1$ are both differences of two points, and satisfies the  following familiar properties:
\begin{enumerate}
 \item It is bilinear: for all degree zero divisors $\sD_0$, $\sD_1$ and $\sD_2$
 \[
 [\sD_0; \sD_1+\sD_2] = [\sD_0;\sD_1]\times [\sD_0,\sD_2], \quad
 [\sD_0+\sD_1;\sD_2] = [\sD_0; \sD_2]\times [\sD_1;\sD_2], 
 \]
 \item It is $B_{\C_p}^\times$-equivariant: 
 \[
 [\gamma \sD_0; \gamma \sD_1] = [\sD_0;\sD_1] \quad \mbox{ for all } \gamma\in B_{\C_p}^\times.
 \]
 \item It is symmetric (Weil reciprocity):
 \[ 
 [\sD_0;\sD_1] = [\sD_1;\sD_0].
 \]
 \item Given any pair  $\sD_0$ and $\sD_1$ of degree zero divisors on $\sH_p$ for which the support of $\sD_0$ is disjoint from the $\Gamma$-orbit of the support of $ \sD_1$, the infinite product
 \[
 [\sD_0; \sD_1]_\Gamma \defeq \prod_{\gamma\in \Gamma} [\sD_0; \gamma \sD_1]
 \]
 converges absolutely in $\C_p^\times$ (see page 47 of \cite{GvdP}). 
\end{enumerate}
The quantity $[\sD_0; \sD_1]_\Gamma$ is called the
{\em modular Weil symbol} attached to the divisors $\sD_0$ and $\sD_1$ on $\sH_p$ and to the discrete $p$-arithmetic group $\Gamma$. It can be used to describe the Jacobian of $X$ as follows: let $L$ be a complete extension of $\Q_{p^2}$, $\sD \in \Div^0(\sH_p(L))$ a divisor of degree $0$ and choose $\eta \in \sH_p(L)$ such that $(\eta)$ and $\alpha \sD$ have disjoint support for all $\alpha \in \Gamma$. Then, define $\theta_\sD$ via
\[
\theta_\sD(z) = [(z) - (\eta) ; \sD]_\Gamma\quad \forall z \in \sH_p(L).
\]
Note that for $\gamma \in \Gamma$ one gets
\[
\frac{\theta_{\sD}(\gamma z)}{\theta_{\sD}(z)} = [(\gamma z) - (z); \sD]_\Gamma = [(\gamma \eta) - (\eta); \sD]_{\Gamma} \in L^\times,
\]
where in the last equality we used that the modular Weil symbol is invariant under the action of $\Gamma$ on any of the two divisors, and therefore the second expression is independent of $z$.
We then denote
\begin{equation}\label{eqn:def-jD(gamma)}
j_{\sD}(\gamma) = [(\gamma \eta) - (\eta) ; \sD ]_\Gamma.
\end{equation}
The function $j_{\sD}$ defines an element in $\Hom(\Gamma , L^\times) = \Hom(\Gamma_{\mathrm{ab}} , L^\times)$, where ${\Gamma}_{\mathrm{ab}}$ is the abelianization of $\Gamma$. In fact, $j_\sD$ factors through the maximal torsion-free quotient of $\Gamma_{\mathrm{ab}}$, that will be denoted by $\bar{\Gamma} = \Gamma_{\mathrm{ab}}/(\Gamma_{\mathrm{ab}})_{\mathrm{tors}}$.
We need to introduce one more ingredient, the so-called {\em $p$-adic period pairing}. Define 
\[
\langle \ , \ \rangle\colon \Gamma\times \Gamma \rightarrow \C_p^\times,
\]
by choosing arbitrary base points $\tau_1,\tau_2\in \sH_p$ that are not $\Gamma$-equivalent and
setting 
$$ \langle \gamma_1, \gamma_2\rangle \defeq [(\gamma_1\tau_1)-(\tau_1); (\gamma_2\tau_2)-(\tau_2)]_\Gamma.$$
In a similar way as above, it can be seen that this expression does not depend on the choice of $\tau_1$ and $\tau_2$, and is a homomorphism on each argument. Moreover, it descends to a pairing $ \langle \ , \ \rangle\colon \bar{\Gamma} \times \bar{\Gamma} \to \Q_p^\times$. The group $\bar{\Gamma}$ is finitely generated of rank equal to the genus $g$ of the Shimura curve $X$ and the pairing $\langle \ , \ \rangle$  gives  an embedding
\[ 
j\colon \bar{\Gamma} \intoo  \Hom(\bar{\Gamma}, \Q_p^\times) \simeq (\Q_p^\times)^g.
\]
Now, for a given $\Delta \in \Div^0(\Gamma \backslash \sH_p(L))$, choose $\sD\in \Div^0(\sH_p(L))$ such that $\pi_\ast \sD = \Delta$ and define
\[
\AJ\colon {\Div}^0(\Gamma\backslash \sH_p(L)) \too \Hom(\bar{\Gamma}, L^\times)/j(\bar{\Gamma}), \ \Delta \mapstoo [j_\sD].
\] 
It is a calculation to verify that the equivalence class of $j_\sD$ is independent of the choice of lift of $\Delta$, showing that the map $\AJ$ is well-defined. 
Remember that $J$ denotes the Jacobian of the curve $X$.
\begin{proposition}\label{prop: description p-adic points of Jacobian}
	The map $\AJ$ defined above is trivial on the group of principal divisors and, for every complete extension $L$ of $\Q_{p^2}$, it induces an identification
	\[
	J(L) \simeq \Hom(\bar{\Gamma}, L^\times)/j(\bar{\Gamma}).
	\]
	Moreover, if $L/\Q_{p^2}$ is a Galois extension, the identification is $\Gal(L/\Q_{p^2})$-equivariant.
\end{proposition}
\begin{proof}
	See VI.2. and VIII.4 of \cite{GvdP}.
\end{proof}
In view of the previous proposition, $\AJ$ can be interpreted as a $p$-adic Abel--Jacobi map.
We also note that, by the positive definiteness of the pairing $\ord_p \circ \langle \cdot, \cdot \rangle$, the natural homomorphism from $\Hom(\bar{\Gamma}, \Z_{p^2}^\times)$ to $\Hom(\bar{\Gamma}, \Q_{p^2}^\times)/j({\Gamma})$ is an injection, whose image has finite index.
This gives the explicit description
\[
J(\Q_{p^2})_\Q \simeq H^1(\bar{\Gamma}, \Z_{p^2}^\times)_\Q.
\]

\subsection{Divisors of strong degree $0$}\label{subsec: div of strong degree 0} For any vertex $L \in \cT_0$, consider the affinoid $\cA_L \defeq \red^{-1}(L) \subset \sH_p$ and the wide open $\cW_L \subset \sH_p$ given as the preimage by $\red$ of the union of the vertex $L$ and all the (open) edges of $\cT$ that have $L$ as one of its endpoints.
\begin{definition}\label{def: divisor of strong degree 0}
	Let $\sD$ be a finite divisor on $\sH_p$. 
	\begin{enumerate}
		\item $\sD$ is of {\em strong degree $0$ in an even sense} if for every $L \in \cT_0$ even vertex (resp.~odd vertex), we have that $\sD \cap \cW_L$ (resp.~$\sD \cap \cA_L$) is of degree $0$.
		\item  $\sD$ is of {\em strong degree $0$ in an odd sense} if for every $L \in \cT_0$ odd vertex (resp.~even vertex), we have that $\sD \cap \cW_L$ (resp.~$\sD \cap \cA_L$) is of degree $0$.
	\end{enumerate}
	A divisor $\Delta \in \Div(\Gamma \backslash \sH_p)$
	is of {\em of strong degree zero} if the following equivalent conditions hold: 
	\begin{enumerate}
		\item There exists divisors $\sD_e, \sD_o \in {\Div}(\sH_p)$ of strong degree $0$ in an even and odd sense respectively such that $\pi_\ast(\sD_e) = \pi_\ast(\sD_o) = \Delta$.
		\item The formal divisor $\hat{\sD} = \pi^\ast \Delta$ satisfies that, for every $L \in \cT_0$, the divisors $\hat{\sD} \cap \cW_L \mbox{ and } \hat{\sD} \cap \cA_L$ have degree $0$.
	\end{enumerate}
\end{definition}

We denote by ${\Div}^{0}_s(\Gamma \backslash \sH_p)$ the group of divisors of strong degree $0$ on $\Gamma \backslash \sH_p$. We also denote  ${\Div}^{0}_{s,e}(\sH_p)$ (resp.~$\Div^{0}_{s,o}(\sH_p)$) the group of divisors of strong degree $0$ on $\sH_p$ in an even (resp.~odd) sense. The motivation for these notions is explained in the next lemma. 

\begin{lemma}\label{lemma: lift of AJ divisor strong degree 0}
Let $\Delta$ be an element of ${\Div}^{0}_s(\Gamma \backslash \sH_p)$.
The homomorphism $j_{\sD_e}\in\Hom(\bar{\Gamma}, \C_p^\times)$ does not depend on a choice of $\sD_e \in \Div^{0}_{s,e}(\sH_p)$ with $\pi_\ast(\sD_e) = \Delta$.
In particular, the morphism
	\[
	{\Div}^{0}_s(\Gamma \backslash \sH_p) \too \Hom(\bar{\Gamma}, \C_p^\times), \ \Delta \mapstoo j_{\sD_e},
	\]
	is a well-defined lift of the restriction of $\AJ$ to $\Div_s^0(\Gamma \backslash \sH_p)$.
	The same is true if one replaces $e$ by $o$ everywhere.
\end{lemma}
\begin{proof}
Let $\sD$, $\sD' \in \Div^{0}_{s,e}(\sH_p)$ be such that $\pi_\ast \sD = \pi_\ast \sD' = \Delta$.
By the strong degree $0$ assumption there exist vertices $L_1, \dots, L_r \in \cT_0$ and a decomposition
	\[
	\sD = \sD_1 + \dots + \sD_r
	\]
	such that for $1\leq i\leq r$ the divisor $\sD_i$ is of degree $0$ and supported on
	\begin{enumerate}
		\item $\cW_{L_i}$ if $L_i$ is an even vertex, or
		\item $\cA_{L_i}$ if $L_i$ is an odd vertex.
	\end{enumerate}
	Since $j_{\sD_j} = j_{\gamma \sD_j}$, for every $\gamma \in \Gamma$, we can suppose that the vertices $L_1, \dots, L_r$ are not $\Gamma$-equivalent.
	Proceeding similarly for $\sD'$, there exist lattices $L_1', \dots , L_{r'}' \in \cT_0$ and degree $0$ divisors $\sD_1', \dots, \sD_{r'}' \in \Div^0(\sH_p)$ satisfying the same conditions as above.
	We have 
	\[
	\hat{\sD} = \sum_{i = 1}^r\sum_{\alpha \in \Gamma} \alpha {\sD}_i = \sum_{i = 1}^{r'}\sum_{\alpha \in \Gamma } \alpha {\sD}_i'.
	\]
	For $\alpha \in \Gamma$, the divisor $\alpha {\sD}_i$ has support in $\cW_{\alpha L_i}$ if $L_i$ is even and has support in $\cA_{\alpha L_i}$ if $L_i$ is odd. Note that these supports are disjoint when $i$ varies from $1$ to $r$ and $\alpha$ varies over $\Gamma$, as $\Gamma$ does not stabilize any vertex because $\bar{\Gamma}$ is torsion-free. Moreover, the same holds for the divisors $\alpha{\sD}_i'$. Thus, we conclude that $r = r'$ and there exist $\alpha_1, \dots, \alpha_r \in \Gamma$ such that 
	\[
	{\sD}_i = \alpha_i{\sD}_i'
	\]
	for every $i$ (after rearranging terms, if needed).
	We therefore have that $j_{\sD_i} = j_{\sD_i'}$ for all $i$ and the equality $j_{\sD} =j_{\sD'}$ follows.	
\end{proof}

\begin{remark}
	If $\Delta \in \Div_s^0(\Gamma \backslash \sH_p)$ is a divisor supported on preimages of vertices by the reduction map, both lifts $\sD_e$ and $\sD_o$ are divisors of strong degree $0$ in an even sense and in an odd sense simultaneously. We will sometimes drop the subindices $e$ and $o$ in this case.
\end{remark}

\subsection{Reduction of the main theorem to convenient Schwartz--Bruhat functions}\label{subsec: reduction of theorem to convenient Schwartz--Bruhat functions} 
Recall the action of $\T^N$ on $\hat{R}^\times$-invariant Schwartz--Bruhat functions introduced in Section \ref{section: Cerednik--Drinfeld theorem}.
We can similarly define an action of $\T^N$ on the space $\mathrm{Funct}(\Gamma \backslash \cT_0, \Z)$ of $\Gamma$-invariant integral functions on $\cT_0$.
\begin{definition}
	A Schwartz--Bruhat function $\Phi$ on $V _{\A^{p,\infty}}$ is {\em convenient} if it is special, $\Phi(0) = 0$, and for every $D \in \cD_S$ the divisor $\Delta_\Phi(D)$ is of strong degree $0$.
\end{definition}

\begin{lemma}\label{lemma: div strong degree 0 and Hecke operators}
	Let $\Phi$ be a special Schwartz--Bruhat function and let $T \in \T^N$ be a Hecke operator that annihilates the space $\mathrm{Funct}(\Gamma \backslash \cT_0, \Z)$. Then, the Schwartz--Bruhat function $T(\Phi)$ is convenient. 
\end{lemma}
\begin{proof}
For $L \in \cT_0$, denote by $\delta_L$ the characteristic function of $L$.
	Define the homomorphism
	$\deg_{\cT_0}\colon \Div(\sH_p) \to \mathrm{Funct}(\cT_0, \Z)$ by
	\[
	\deg_{\cT_0}((P)) = \begin{cases}
	\delta_L, \mbox{ if } \red(P) = L \in \cT_0, \\
	\delta_L + \delta_{L'} \mbox{ if } \red(P) = (L, L') \in \cT_1.
	\end{cases}
	\]
	This morphism is $B^\times$-equivariant, hence induces a $\T^N$-equivariant morphism
	\[
	\deg_{\cT_0}\colon \Div(\Gamma \backslash \sH_p) \too \mathrm{Funct}(\Gamma \backslash \cT_0, \Z).
	\]
	We proceed to verify that $\Delta_{T(\Phi)}(D)$ is of strong degree $0$ for a fixed $D \in \cD_S$. From the Hecke equivariance of $\deg_{\cT_0}$, we have
	\[
	\deg_{\cT_0}(\Delta_{T(\Phi)}(D)) = \deg_{\cT_0}(T(\Delta_\Phi(D))) = T(\deg_{\cT_0}(\Delta_\Phi(D))) = 0,
	\]
	where we used Lemma \ref{lemma: Hecke action divisors and Schwartz--Bruhat functions} in the first equality. Since $\Delta_{T(\Phi)}(D)$ is supported on preimages of vertices (resp.~edges) if $\ord_p(D)$ is even (resp.~odd), the fact that $\deg_{\cT_0}(\Delta_{T(\Phi)}(D)) = 0$ implies that $\Delta_{T(\Phi)}(D)$ is of strong degree $0$. Finally, from the fact that $T$ sends the constant functions on $\mathrm{Funct}(\Gamma \backslash \mathcal T_0, \Z)$ to 0, it follows that $(T(\Phi))(0) = 0$.
\end{proof}

Let $G_\Phi(q) \in \Pic(\Gamma \backslash \sH_p) [[q]]$ be the generating series introduced in Theorem \ref{thm:main p-adic}.
We now use the Jacquet--Langlands correspondence to justify that to prove Theorem \ref{thm:main p-adic} it is enough to prove it for the particular case where $\Phi$ is convenient.

\begin{proposition} \label{prop:reduce-to-convenient}
	The following statements are equivalent:
	\begin{enumerate}
		\item The generating series $G_{\Phi}(q)$ is a modular form of weight $3/2$ and level $\Gamma_0(4N)$ for every special Schwartz--Bruhat function $\Phi$.
		\item The generating series $G_{\Phi}(q)$ is a cusp form of weight $3/2$ and level $\Gamma_0(4N)$ for every convenient Schwartz--Bruhat function $\Phi$.
	\end{enumerate}
\end{proposition}
\begin{proof}
	Clearly (1) implies (2). We justify the reverse implication. By Jacquet--Langlands, we have:
	\begin{itemize}
		\item The action of $\T^N$ on $\mathrm{Funt}(\Gamma \backslash \cT_0, \Z)$ factors through the action of the Hecke algebra (away from $N$) on $M_2(\Gamma_0(N/p),\Q)$. 
		\item The action of $\T^N$ on $J(\C_p)_\Q$ factors through the action of the Hecke algebra (away from $S$) on the space of forms in $S_2(\Gamma_0(N), \Q)$ that are new at $p$.
	\end{itemize}
	Let $\ell$ be a prime such that $\ell \not \in S$ and denote by $T_\ell \in T^S$ the corresponding Hecke operator. From the second point, we deduce that we have an isomorphism  
	\[
	\Pic(X)(\C_p)_\Q \xlongrightarrow{ \sim } J(\C_p)_\Q \oplus \Q, \ [\Delta] \mapstoo \left( (T_\ell - \ell - 1)\Delta, \deg(\Delta)\right).
	\]
	Let $\Phi$ be a special Schwartz--Bruhat function. Since we proved that $\deg(G_\Phi)$ is a modular form in Proposition \ref{prop: degGPhi is modular}, after replacing $\Phi$ by $(T_\ell - \ell - 1)(\Phi)$ (and by Lemma \ref{lemma: Hecke action divisors and Schwartz--Bruhat functions}) we may suppose that $G_\Phi(q) \in J(\C_p)_\Q [[q]]$. 	
	Now, choose $T \in \T^N$ satisfying 
	\begin{itemize}
		\item $T$ annihilates $\mathrm{Funct}(\Gamma \backslash \cT_0, \Z)$ and 
		\item $T\colon J(\C_p)_\Q  \to J(\C_p)_\Q$ is a bijection.
	\end{itemize}
	By the first property and Lemma \ref{lemma: div strong degree 0 and Hecke operators}, $T(\Phi)$ is convenient and therefore $G_{T(\Phi)}(q) = T(G_\Phi(q))$ is a modular form by hypothesis. Here $T(G_\Phi(q))$ denotes the $q$-expansion obtained by applying $T$ to each of the coefficients of $G_\Phi(q)$.
	The fact that $G_\Phi(q) \in J(\C_p)_\Q [[q]]$ is a modular form follows from the bijectivity of $T$ on the Jacobian.
\end{proof}

Let $\Phi$ be a convenient Schwartz--Bruhat function. For every $D \in \cD_S$, fix $\sD_{\Phi}(D)_e \in \Div^{0}_{s,e}(\sH_p)$ and $\sD_{\Phi}(D)_o \in \Div^{0}_{s,o}(\sH_p)$ such that $\pi_\ast \sD_{\Phi}(D)_e = \pi_\ast \sD_{\Phi}(D)_o  = \Delta_\Phi(D)$. Note that such lifts can be chosen such that they are invariant under the action of $\Aut(\C_p/\Q_p)$.
It follows from there that the homomorphisms $j_{\sD_\Phi(D)_e}$ and $j_{\sD_\Phi(D)_o}$ take values in $\Q_p^\times$.
Consider the generating series 
\[
G_\Phi^+(q) = \sum_{D \in \cD_S} j_{\sD_\Phi(D)_e} \cdot j_{\sD_\Phi(D)_o} q^D \in \Hom(\bar{\Gamma}, {\Q}_p^\times)[[q]].
\]
Since $\Phi$ is special, we have that $a_D(G_\Phi^+(q)) = a_{Dp^{2n}}(G_\Phi^+(q))$ for all $n\geq 0$.
Modularity of $G_\Phi^+(q)$ clearly implies the modularity of $G_\Phi(q)$.
Thus, we will only consider the former in the remainder of this article.
 
\section{Values of $p$-adic theta functions}\label{section: values of p adic theta functions}

The goal of this section is to give an explicit expression for the quantity $j_\sD(\gamma)$ when $\gamma \in \Gamma$ is hyperbolic at $p$ and $\sD$ is a divisor on $\Gamma\backslash\sH_p$ of strong degree $0$. The formulas we will present have a similar flavor to the ones for toric values of lifting obstructions of rigid meromorphic cocycles given in \cite[Section 5.3]{DV--borcherds}. There, the orthogonal group of signature $(3,0)$ is replaced by an orthogonal group of signature $(1,2)$.

Fix an element $\gamma \in \Gamma$ that is hyperbolic at $p$.
It has two distinct fixed points
\[
\xi^+, \xi^-  \in  C_V(\Q_p)
\]
on the boundary of $\sH_p$.
 We order them in such a way that $\xi^+$ and $\xi^-$ are the attractive and repulsive fixed points of $\gamma$, i.e., 
\[
\lim_{M\rightarrow +\infty} \gamma^M \tau = \xi^+, \quad
\lim_{M\rightarrow -\infty} \gamma^M \tau = \xi^-,
\]
for all $\tau\in \sH_p$.

 \begin{lemma}\label{lemma: c_D in terms of a product over cosets of Gamma by gamma}
 For every $\sD \in \Div^0(\sH_p)$ the following equality holds:
 \[
 j_\sD(\gamma) = \prod\limits_{\alpha \in \gamma^{\mathbb{Z}}\backslash\Gamma}\left[ (\xi^+)-(\xi^-); \alpha \sD\right].
 \]
 \end{lemma}
 \begin{proof}
Recall first from \eqref{eqn:def-jD(gamma)}
that
\[
j_\sD(\gamma) = \prod\limits_{\alpha \in \Gamma} \left[(\alpha \gamma \tau)-(\alpha\tau); \sD \right],
\]
where $\tau$ is an arbitrary base point in $\sH_p$. Since this infinite product  converges absolutely, it can be rearranged by grouping together the factors that belong to the same 
coset for $\gamma^\Z$ in $\Gamma$
\[
j_\sD(\gamma) = \prod\limits_{\alpha\in \Gamma/\gamma^{\mathbb{Z}}} \left( \prod\limits_{i=-\infty}^{\infty} \left[(\alpha\gamma^{i+1}\tau)-(\alpha\gamma^{i}\tau); \sD \right] \right).
\]
The innermost product on the right hand side is equal to
\begin{equation*}
\begin{split}
 \lim\limits_{M\to \infty} \prod\limits_{i=-M}^M \left[(\alpha\gamma^{i+1}\tau)-(\alpha\gamma^{i}\tau); \sD\right] & = \lim\limits_{M\to \infty} \left[(\alpha\gamma^{M+1} \tau)-(\alpha\gamma^{-M}\tau); \sD\right] \\ 
 & =   \left[(\alpha\xi^+) -(\alpha\xi^-); \sD\right].
 \end{split}
 \end{equation*}
It follows that
$$
j_\sD(\gamma)= \prod\limits_{\alpha\in \Gamma/\gamma^{\mathbb{Z}}}\left[
  (\alpha\xi^+)-(\alpha\xi^-); \sD\right] =  \prod\limits_{\alpha \in \gamma^{\mathbb{Z}}\backslash\Gamma}\left[ (\xi^+)-(\xi^-); \alpha \sD\right],$$
  where the last equation was obtained by substituting $\alpha$ for $\alpha^{-1}$ and 
  exploiting the fact that the Weil Symbol is $B_{\Q_p}^\times$-equivariant.
  \end{proof}

\subsection{The quotient $\gamma^\Z \backslash \cT$}
  
We will rewrite the infinite product of Lemma \ref{lemma: c_D in terms of a product over cosets of Gamma by gamma} by making an explicit choice of coset representatives for $\gamma^\Z$ in $\Gamma$, well adapted to the calculation at hand. To make this choice, we will exploit the action of $\gamma^\Z$ on the Bruhat-Tits tree $\cT$. We explain some of the properties of such action.
  
The  element $\gamma\in \Gamma$ is hyperbolic at $p$, and acts on $V_{\Q_p}$ with three distinct eigenvalues $\varpi, 1$, and $ \varpi^{-1}$, where $\varpi$ is a global $p$-unit  of norm $1$ in the quadratic imaginary field that splits the characteristic polynomial of $\gamma$ (relative to an embedding of this quadratic imaginary field into $\Q_p$). The valuation $\ord_p(\varpi) = 2t>0$ is an even integer. Letting  $V[\lambda]$ denote the eigenspace in $V$ on which $\gamma$ acts as multiplication by $\lambda$, one obtains the decomposition
\begin{equation} \label{eqn:eigen-decomp}
	V_{\Q_p} = V[ {\varpi}] \oplus V[{\varpi^{-1}}] \oplus V[1].
\end{equation}
The first two eigenspaces are isotropic and together generate a hyperbolic plane in $V_{\Q_p}$ whose orthogonal complement is $V[1]$. The fixed points $\xi^+$ and $\xi^-$ of $\gamma$  in $C_V(\Q_p)$ correspond to the isotropic lines $V[\varpi]$ and $V[{\varpi^{-1}}]$ respectively. Given a $\Z_p$-lattice $L\subset V_{\Q_p}$, 
the eigenspace decomposition  \eqref{eqn:eigen-decomp} induces a containment
\[ L \supset L[\varpi] \oplus L[{\varpi^{-1}}] \oplus L[1], 
\]
where 
$$
L[\varpi] \defeq L \cap V[\varpi], \qquad L[{\varpi^{-1}}] = L \cap V[{\varpi^{-1}}], \qquad
L[1] = L \cap V[1].$$
\begin{definition}\label{def: depth of a lattice}
	The {\em depth} of $L$ with respect to $\gamma$ is the integer $n$ such that 
	\[
	p^{2n} = [L :L[\varpi] \oplus L[{\varpi^{-1}}] \oplus L[1]].
	\]
\end{definition} 
The depth measures how far $L$ is from admitting an eigenspace 
decomposition under $\gamma$ as modules over $\Z_p$. Lattices that are of depth $0$ are precisely  those that decompose into a direct sum  of eigen-submodules for $\gamma$.
If $L= L[\varpi]\oplus L[1] \oplus L[\varpi^{-1}]$ is of depth zero, then the same is true of the lattices 
\[
L_j = p^j L[\varpi] \oplus L[1] \oplus p^{-j} L[\varpi^{-1}], \qquad j \in \Z.
\]
The unimodular lattices $L_j$ and $L_{j+1}$ are $p$-neighbours, and the element $\gamma$ sends $L_j$ to $L_{j+2t}$. After fixing the base lattice $L_0$, the sequence of successive $p$-neighbours
\[
g_\gamma = \big\{ \ldots, L_{-2}, L_{-1}, L_0, L_1, L_2, L_3, \ldots  \big\}
\]
determines an infinite geodesic on $\cT$ which is globally preserved by $\gamma$. We suppose that the enumeration is done so that $L_0$ is an even vertex of $\cT$.

\begin{definition}\label{remark: relation depth with distance}
	Let $L\subseteq V_{\Q_p}$ be a unimodular $\Z_p$-lattice in $V_{\Q_p}$.
	The lattice $L_{i}\in g_\gamma$ that is closest to $L$ is called the {\em parent} of $L$.
\end{definition}
	The distance from $L$ to its parent is equal to the depth of $L$.
A fundamental region  for $\gamma^\Z\backslash\cT_0$ can therefore be defined by setting
\[
\cT_{0,\gamma} \defeq \big\{ L\in \cT_0 \mbox{ with } \mathrm{Parent}(L) \in \{ L_0, L_1, L_2, \ldots L_{2t-1} \}   \big\}.
\]
The subset $\cT_{0, \gamma }\subset \cT$ can be written as an increasing union of finite subsets
\[
\cT_{0, \gamma}  =  \bigcup_{n\ge 0} \cT_{0, \gamma}^{\le n}, \qquad 
\cT_{0, \gamma}^{\le n} \defeq \{ L \in \cT_{0, \gamma}  \mbox{ with } \mathrm{depth}(L) \le n \}.
\]
Let $\cA_{\gamma}$ respectively $\cA_{\gamma}^{\le n}$ be the subsets of $\sH_p$ given as the preimages of $\cT_{0, \gamma}$ respectively $\cT_{0, \gamma}^{\le n}$ under the reduction map. The set $\cA_{\gamma}$ can thus be expressed as an increasing union of affinoid subsets,
\begin{equation}\label{eq: A 0 gamma as an increase union}
\cA_{\gamma} = \bigcup_{n\ge 0} \cA_{\gamma}^{\le n}.
\end{equation}

Using $\cT_{0, \gamma}$, we proceed to give several fundamental regions for $\gamma^\Z \backslash \cT_1$. Define $\cT_{1, \gamma, e}$ to be the set of edges in $\cT_1$ such that its even vertex is in $\cT_{0, \gamma}$. For a given vertex $L \in \mathcal T_0$, let $W_L \subset \mathcal T_1$ be the set of open edges that have $L$ as one of its endpoints. We then have, 
\[
\cT_{1, \gamma, e} = \bigcup_{n \geq 0} \cT_{1, \gamma , e}^{\leq n}, \qquad \cT_{1,\gamma, e}^{\leq n} := \bigcup_{\substack{L \text{ even} \\ L \in \cT_{0,\gamma}^{\leq n} }} W_L.
\]
Let $\cW_{\gamma , e}$ and $\cW_{\gamma , e}^{\leq n}$ be the preimage by $\red$ of $\cT_{1, \gamma , e}$ and $\cT_{1, \gamma, e}^{\leq n}$, respectively. We then have
\begin{equation}\label{eq: W 1 gamma e as an increase union}
	\cW_{\gamma, e} = \bigcup_{n \geq 0} \cW_{\gamma , e}^{\leq n}.
\end{equation}
Observe that for every $n$ the set $\cW_{\gamma , e}^{\leq n}$ can be written as the disjoint union of sets of the form $\cW_L - \cA_L$, where $L$ runs over even vertices in $\cT_{0, \gamma}^{\leq n}$.
Similarly, define $\cT_{1 , \gamma, o}$, $\cT_{1, \gamma , o}^{\leq n}$, $\cW_{\gamma , o}$ and $\cW_{\gamma , o}^{\leq n}$ by replacing even by odd everywhere.

\subsection{Computation of $j_\sD(\gamma)$ for divisors of strong degree $0$}
With the notations given in the previous section in place, we can prove the following formulas. 
\begin{proposition}\label{prop:theta-gamma-div}
Let $\sD$ be a divisor on $\sH_p$ of strong degree zero supported on preimages of vertices of $\cT$ under the reduction map. Let $\hat{\sD} = \sum_{\alpha \in \Gamma} \alpha \sD \in {\Div}^\dag(\sH_p)$. Then, 
\[
j_\sD(\gamma) = \lim_{n\rightarrow \infty} \left[(\xi^+) -(\xi^-); \hat\sD\cap \cA_{\gamma}^{\le n}\right].
\]
\end{proposition}
\begin{proof}
	Since $\sD$ is of strong degree $0$ we can write $\sD = \sum_{i = 1}^r \sD_{L_i}$, where $\sD_{L_i}$ is a degree $0$ divisor supported on $\cA_{L_i}$ and the $L_i$ are vertices in $\cT_0$. By Lemma \ref{lemma: c_D in terms of a product over cosets of Gamma by gamma}, we have
	\begin{equation}\label{eq: calculation jDgamma}
	j_\sD(\gamma) = \prod_{\alpha \in \gamma^\Z \backslash \Gamma}\left[(\xi^+) - (\xi^-); \alpha \sD\right] = \prod_{i = 1}^r \prod_{\alpha \in \gamma^\Z \backslash \Gamma} \left[(\xi^+) - (\xi^-); \alpha \sD_{L_i}\right].
	\end{equation}
	Now, it follows from the definition of $\cT_{0,\gamma}$, that for every $i \in \{1, \dots, r\}$ and for every class $[\alpha] \in \gamma^\Z \backslash \Gamma$, there is precisely one representative $\gamma^{k_{i, \alpha}} \alpha$ such that $\gamma^{k_{i, \alpha}} \alpha \sD_{L_i}$ is supported on $\cA_{\gamma}$. This implies that if we write
	\[
	\hat{\sD} = \sum_{\alpha \in \Gamma} \alpha \sD = \sum_{i = 1}^r \sum_{\alpha \in \gamma^\Z \backslash \Gamma} \sum_{k = -\infty}^{+\infty}\gamma^k \alpha \sD_{L_i},
	\]
	we have 
	\[
	\hat{\sD} \cap \cA_{\gamma} = \sum_{i = 1}^r \sum_{\alpha \in \gamma^\Z\backslash \Gamma} \gamma^{k_{i, \alpha}} \alpha \sD_{L_i}.
	\]
	On the other hand, using \eqref{eq: calculation jDgamma} and that $\xi^+$ and $\xi^-$ are fixed by $\gamma$ we can write  
	\[
	j_\sD(\gamma) = \prod_{i = 1}^r \prod_{\alpha \in \gamma^\Z\backslash \Gamma} \left[(\xi^+) - (\xi^-); \gamma^{k_{i, \alpha}} \alpha \sD_{L_i}\right].
	\]
	Combining the last two equalities, and specifying the order of multiplication on the last expression for $j_\sD(\gamma)$ given by the increasing union of \eqref{eq: A 0 gamma as an increase union}, we obtain the desired result.
\end{proof}

We can obtain similar expressions for $j_\sD(\gamma)$ when $\sD$ is a divisor of strong degree $0$ supported on preimages of edges by the reduction map. 

\begin{proposition}\label{prop:theta-gamma-div edges}
	Let $\sD$ be a divisor on $\sH_p$ supported on preimages of edges of $\cT$ by the reduction map. Let $\hat{\sD} = \sum_{\alpha \in \Gamma} \alpha \sD \in {\Div}^\dag(\sH_p)$. We have:
	\begin{enumerate}
		\item If $\sD = \sD_e$ is of strong degree $0$ in an even sense, then 
		\[
		j_{\sD_e}(\gamma)  = \lim_{n \to +\infty} \left[(\xi^+) - (\xi^-); \hat{\sD} \cap \cW_{\gamma, e}^{\leq n}\right].
		\] 
		\item If $\sD = \sD_o$ is of strong degree $0$ in an odd sense, then 
		\[
		j_{\sD_o}(\gamma)  = \lim_{n \to +\infty} \left[(\xi^+) - (\xi^-); \hat{\sD} \cap \cW_{\gamma, o}^{\leq n} \right].
		\]
	\end{enumerate}
\end{proposition}
\begin{proof}
	We only give the proof for the first case, the second being similar. Write $\sD = \sum_{i = 1}^r \sD_{L_i}$, where $\sD_{L_i}$ is a degree $0$ divisor supported on $\cW_{L_i} - \cA_{L_i}$ and the set $\{L_1, \dots, L_r\}$ consists on even vertices of $\cT_0$. Hence, we have
	\[
	j_\sD(\gamma) = \prod_{\alpha \in \gamma^\Z \backslash \Gamma}\left[(\xi^+) - (\xi^-); \alpha \sD\right] = \prod_{i = 1}^r \prod_{\alpha \in \gamma^\Z\backslash \Gamma}\left[(\xi^+) - (\xi^-); \alpha \sD_{L_i}\right].
	\]
	Now, for every class $[\alpha] \in \gamma^\Z \backslash \Gamma$ and $L_i$ even vertex as above, there exists precisely one representative $\gamma^{k_{i, \alpha}} \alpha$ such that $\gamma^{k_{i, \alpha}} \alpha L_i \in \cT_{0, \gamma}$. It follows from there that the divisor $\gamma^{k_{i , \alpha}} \alpha \sD_{L_i}$ is supported on $\gamma^{k_{i, \alpha}} \alpha \cW_{L_i} -  \gamma^{k_{i, \alpha}} \alpha \cA_{L_i} = \cW_{\gamma^{k_{i, \alpha}} \alpha L_i} - \cA_{\gamma^{k_{i, \alpha}} \alpha L_i} \subset \cW_{\gamma, e}$. This implies that if 
	\[
	\hat{\sD} = \sum_{i = 1}^r \sum_{\alpha \in \gamma^\Z\backslash \Gamma} \sum_{k = -\infty}^{+\infty} \gamma^k \alpha \sD_{L_i},
	\]
	we have 
	\[
	\hat{\sD} \cap \cW_{\gamma, e} = \sum_{i = 1}^r \sum_{\alpha \in \gamma^\Z \backslash \Gamma} \gamma^{k_{i, \alpha}} \alpha \sD_{L_i}.
	\]
	On the other hand,  
	\begin{equation}\label{eq: middleexpressioncDeltaedges}
		j_\sD(\gamma) = \prod_{\alpha \in \gamma^\Z \backslash \Gamma}\left[(\xi^+) - (\xi^-); \gamma^{k_{i, \alpha}}\alpha \sD_{L_i}\right].
	\end{equation}	
	Note that $\cW_{\gamma,e}^{\leq n}$ can be written as a union of sets of the form $\cW_{L_i} - \cA_{L_i}$, where the union is over even vertices in $\cT_{0, \gamma}^{\leq n}$. Hence, $\hat{\sD} \cap \cW_{\gamma,e}^{\leq n}$ is a degree $0$ divisor (because $\sD$ is of strong  degree $0$). Moreover, the increasing union over $n$ of the sets $\cW_{\gamma,e}^{\leq n}$ covers $\cW_{\gamma,e}$, as we deduced in \eqref{eq: W 1 gamma e as an increase union}. This implies that we can use these sets to specify an order of multiplication on \eqref{eq: middleexpressioncDeltaedges} to obtain the desired expression.
\end{proof}

\section{Abel--Jacobi images of Heegner divisors}\label{section: AJ images of Heegner divisors}
We use the results of Section \ref{section: values of p adic theta functions} to compute Abel--Jacobi images of the Heegner divisors introduced in Section \ref{section: Cerednik--Drinfeld theorem}. More precisely, let $\Phi$ be a convenient Schwartz--Bruhat function and fix $D \in \cD_S$. Choose $\sD_\Phi(D)_e, \sD_\Phi(D)_o$ divisors on $\sH_p$ of strong degree $0$ in an even and odd sense respectively that lift $\Delta_\Phi(D)$ and let $\hat{\sD}_\Phi(D) = \pi^\ast\Delta_\Phi(D)$.
If $\Delta_\Phi(D)$ is supported on preimages of vertices under the reduction map, we suppose that $\sD_\Phi(D)_e =  \sD_\Phi(D)_o$ and we will drop the subindices $e$ and $o$.
At last, let $\gamma\in \Gamma$ be an element hyperbolic at $p$.
We will compute $j_{\sD_\Phi(D)_e}\cdot j_{\sD_\Phi(D)_o} (\gamma)$.

\subsection{Values of theta functions associated to Heegner divisors}\label{section: First part of computation j DPhi(D)}
Since $\Phi$ is invariant under multiplication by $p$, we have $\Delta_\Phi(D) = \Delta_\Phi(Dp^{2n})$ for every $n \geq 0$. Therefore, we will assume here and for the rest of the section that $D$ is an element of $\cD_S$ with $\ord_p(D) \in \{ 0 , 1\}$. 
In view of the notion of depth of a lattice with respect to $\gamma$, which was introduced in Definition \ref{def: depth of a lattice} and Remark \ref{remark: relation depth with distance}, the following definition will be relevant.
\begin{definition}
	Let $v \in V$ be a vector such that $Q(v)=D$.
	The {\em depth} of $v$ with respect to $\gamma$ is
	\[
	\mathrm{depth}(v)  \defeq \min_{\substack{L \ni v}} \{ \mathrm{depth}(L) \},
	\]
	where the minimum is taken over all unimodular $\Z_p$-lattices in $V_{\Q_p}$ such that $v \in L$. 
\end{definition} 
Note that by Lemma \ref{lem: howmanylattices} there are exist at most two unimodular $\Z_p$-lattices containing $v$.
We now present the computation of $j_{\sD_\Phi(D)_e}\cdot j_{\sD_\Phi(D)_o}(\gamma) $, which is slightly different according to the $p$-adic valuation of $D$.
Let $n\geq 1$.
If $\ord_p(D) = 0$, consider
\[
\hat\sD_\Phi(D) \cap \cA_\gamma^{\le n} =    \sum\limits_{ \substack{ Q(v) = D, \\ L_v \in \cT_{0,\gamma}^{\le n}}} \Phi(v) \Delta(v),
\]
where the sum is over the vectors $v\in V$. Since this divisor is of degree zero, as $\Phi$ is convenient, the function on $C_V$ given by
\[
	\xi \mapstoo \prod\limits_{ \substack{ 
	Q(v) = D, \\ L_v \in \cT_{0,\gamma}^{\le n}}} \langle \tilde\xi, v\rangle^{\Phi(v)},
\]
where $\tilde \xi$ is any  vector in the isotropic line generated by $\xi$ in $V_{\C_p}$, is well-defined and has divisor equal to $\hat \sD_\Phi(D)  \cap \cA_\gamma^{\le n}$. Therefore, if $\tilde \xi^+$ and $\tilde \xi^-$ are vectors in $V_{\C_p}$ generating the $\C_p$-lines $\xi^+$ and $\xi^-$ introduced in Section \ref{section: values of p adic theta functions}, Proposition \ref{prop:theta-gamma-div} implies
\begin{equation} \label{eq: theta gamma vertices with depth}
j_{\sD_\Phi(D)}(\gamma) = \lim\limits_{n\rightarrow\infty} 
\prod\limits_{ \substack{ 
		Q(v) = D, \\
		L_v \in \cT_{0,\gamma}^{\le n}}} \left(\frac{\langle \tilde\xi^+, v\rangle}{\langle \tilde\xi^-, v\rangle}\right)^{\Phi(v)} = \lim\limits_{n\rightarrow\infty}  \prod\limits_{ \substack{ v \in \gamma^\Z\backslash V \\ Q(v) = D, \\ \mathrm{depth}(v) \leq n}} \left(\frac{\langle \tilde\xi^+, v\rangle}{\langle \tilde\xi^-, v\rangle}\right)^{\Phi(v)}.
	\end{equation}
Here, the second equality follows from the fact that the terms appearing in the expression of the middle do not change if we replace $v$ by $\gamma v$. If $\ord_p(D) = 1$, we can proceed similarly. In that case, the function on $C_V(\C_p)$ given by
\[
\xi \mapstoo \prod_{\substack{Q(v) = D \\ e_v \in \cT_{1, \gamma, {e}}^{\leq n}}} \langle \xi, v \rangle^{\Phi(v)} \prod_{\substack{Q(v) = D \\ e_v \in \cT_{1, \gamma, {o}}^{\leq n}}} \langle \xi, v \rangle^{\Phi(v)} 
\]
has divisor equal to 
\[
\hat{\sD}_\Phi(D) \cap   \cW_{\gamma, {e}}^{\leq n} + \hat{\sD}_\Phi(D) \cap \cW_{\gamma, {o}}^{\leq n}. 
\]
It then follows from Proposition \ref{prop:theta-gamma-div edges} that
\begin{equation}\label{eq: theta gamma edges with depth}
	\begin{split}
		j_{\sD_\Phi(D)_e}\cdot j_{\sD_\Phi(D)_o}(\gamma) &=  \lim_{n \to +\infty } \prod_{\substack{Q(v) = D \\ e_v \in \cT_{1, \gamma, {e}}^{\leq n}}} \left( \frac{\langle \tilde\xi^+, v \rangle}{\langle \tilde\xi^-, v \rangle}\right)^{\Phi(v)} \prod_{\substack{Q(v) = D \\ e_v \in \cT_{1, \gamma, {o}}^{\leq n}}} \left( \frac{\langle \tilde\xi^+, v \rangle}{\langle \tilde\xi^-, v \rangle}\right)^{\Phi(v)} \\ 
		&= \lim_{n \to +\infty } \prod_{\substack{ v \in \gamma^\Z \backslash V \\ Q(v) = D \\ \mathrm{depth}(v) \leq n-1 }} \left( \frac{\langle \tilde\xi^+, v \rangle}{\langle \tilde\xi^-, v \rangle}\right)^{\Phi(v)} \prod_{\substack{ v \in \gamma^\Z \backslash V \\ Q(v) = D \\ \mathrm{depth}(v) \leq n}} \left( \frac{\langle \tilde\xi^+, v \rangle}{\langle \tilde\xi^-, v \rangle}\right)^{\Phi(v)},
	\end{split}
\end{equation}
where, in the second equality, we used that for every vector $v \in \gamma^\Z \backslash V$ with $Q(v) = D$, the term $ {\langle \tilde\xi^+, v \rangle}^{\Phi(v)}/{\langle \tilde\xi^-, v \rangle}^{\Phi(v)}$ appears two times if the distance from $e_v$ to the geodesic $g_\gamma$ preserved by $\gamma$ is lower or equal than $n-1$, and one time if it is equal to $n$.

\subsection{Vectors of length $D$ in $\gamma^\Z \backslash V$}
Recall that $D \in \mathcal D_S$ is such that $\ord_p(D) \in \{ 0 , 1\}$. We give a concrete choice of representatives of the quotient
\begin{equation}\label{eq: v length D depth leq n mod gamma}
\left\{ v \in \gamma^\Z \backslash V \ \middle| \  Q(v) = D, \ \mathrm{depth}(v) \leq n \right\}.
\end{equation}
This will lead to a relation between $j_{\sD_\Phi(D)}(\gamma)$ and Fourier coefficients of theta series in the next section.

Recall the $\Z_p$-lattices $L_0, L_1, \dots, L_{2t - 1}$ introduced in Section \ref{section: values of p adic theta functions}, which form a set of representatives modulo $\gamma^\Z$ of the vertices in the geodesic of $\cT$ stabilized by $\gamma$. Let $\{w^+ , e , w^-\}$ be generators of $L_0[\varpi]$, $L_0[1]$ and $L_0[\varpi^{-1}]$ respectively. Then, 
\[
w_j^+ = p^j w^+,\  e , \  w_j^- = p^{-j} w^-
\]
are generators of $L_j[\varpi]$, $L_j[1]$ and $L_j[\varpi^-]$ for every $j$. Define
\[
L_{j}^{+}[n] \defeq \left\{ v \in L_j \cap V \ \middle| \  Q(v) = Dp^{2n} \mbox{ and }  \langle v,   w_j^+\rangle  \in \Z_p^\times \right\}
\]
and define $L_j^-[n]$ in a similar way as above but replacing the symbol $+$ by the symbol $-$ everywhere. The motivation for the definition of $L_j^+[n]$ and $L_j^-[n]$ is the following. Let $\cT_j^+[n]$ be the subset of $\cT = \cT_0 \cup \cT_1$ of elements $x \in \cT$ that are at distance equal to $n$ from $L_j$ and satisfy that: 
\begin{enumerate}
	\item If $ x = L$ is a vertex and $\mathrm{Parent}(L) = L_k$, then $k \geq j$.
	\item If $x$ is an edge, for any of its endpoints $L$ we have that if $\mathrm{Parent}(L) = L_k$, then $k \geq j$.
\end{enumerate}
Define $\cT_j^-[n]$ in a similar way but replacing the symbol $\geq$ by the symbol $\leq$ everywhere. It then follows from the description of the action of $\gamma^\Z$ in $\cT$ that the disjoint union
\[
\bigcup_{j = 0}^{2t - 1} \left\{ v \in V \ \middle| \  Q(v) = D , \red(\Delta(v)) \subset \cT_j^+[n]  \right\}
\]
gives a set of representatives of \eqref{eq: v length D depth leq n mod gamma}. Similarly, the same holds if we replace the symbol $+$ by the symbol $-$. 

\begin{lemma}\label{lemma: elementary but key with details}
	The map
	\[
	L_j^+[n] \xlongrightarrow{\sim} \left\{ v \in V \ \middle| \  Q(v) = D , \red(\Delta(v)) \subset \cT_j^+[n]  \right\}, \ u \mapstoo u/p^n
	\]
	is bijective. The same result holds if we replace the symbol $+$ by the symbol $-$ everywhere.
\end{lemma}
\begin{proof}
	We start proving that the map is well-defined. Let $u \in L_j^+[n]$ and let $v = u/p^n$. Since $u$ is primitive, $L_v$ (resp.~$e_v$) is at distance $n$ from $L_j$ if $\ord_p(D) = 0$ (resp.~$\ord_p(D) = 1$). Moreover, the condition $\langle u, w_j^+ \rangle \in \Z_p^\times$ implies that if for any unimodular lattice $L$ containing $v$ we denote $\mathrm{Parent}(L) = L_k$, we have $k \geq j$. Hence, $\red(\Delta(v)) \subset \cT_j^{+}[n]$.
	
	The injectivity of the map is clear, so we are left to prove surjectivity. For that, let $v \in V$ be such that $Q(v) = D$ and $\red(\Delta(v)) \subset \cT_j^+[n]$. Since there is a unimodular $\Z_p$-lattice in $V_{\Q_p}$ containing $v$ at distance $n$ from $L_j$, we have that $p^nv \in L_j$. Note that $\langle p^n v, w_j^+ \rangle \neq 0$. Indeed, for the sake of contradiction suppose that $\langle p^n v, w_j^+ \rangle = 0$. This implies that 
	\[
	p^n v = aw_j^+ + be,
	\]
	for $a, b \in \Z_p$. Then, 
	\[
	\gamma \cdot (p^n v) = a\varpi w_j^+ + be. 
	\]
	Subtracting these two equations, we get that $\gamma v - v \in V$ is either $0$ or it is an eigenvector for the $\Q$-linear action of $\gamma$ on $V$ of eigenvalue $\varpi$. Since $\gamma v - v \in V$ and $\varpi \not \in \Q$, the only possibility is that $\gamma v - v = 0$. This implies that $v \in \langle e \rangle$, giving a contradiction with the fact that $\sqrt{-D} \not \in \Q_p$. We can therefore choose $i \leq j$ such that $p^nv \in L_i^+[n]$. Now, the fact that the map is well-defined applied to the index $i$, together with the observation that the sets $\cT_i^+[n]$ and $\cT_j^+[n]$ are disjoint if $i \neq j$ proves that $i = j$ and we are done.
\end{proof}

We can combine the information of Lemma \ref{lemma: elementary but key with details} for $j = 0, \dots, 2t - 1$ to obtain the following result. 
\begin{proposition}\label{prop: elementary but key}
	Let $n\geq 0$, we have a bijection
	\[
	L_0^+[n] \cup \dots \cup L_{2t - 1}^+[n] \xlongrightarrow{\sim} \left\{ v \in \gamma^\Z \backslash V \ \middle| \  Q(v) = D, \ \mathrm{depth}(v) \leq n \right\}
	\]
	given by $v \mapsto [p^{-n}v]$, where $[p^{-n}v]$ denotes the class of $p^{-n}v \in V$ modulo $\gamma^\Z$. Moreover, the same result is true if we replace the symbol $+$ by the symbol $-$.
\end{proposition}
\begin{proof}
	By Lemma \ref{lemma: elementary but key with details} we have that the map
	\[
	L_0^+[n] \cup \dots \cup L_{2t - 1}^+[n] \xlongrightarrow{\sim} \bigcup_{j = 0}^{2t - 1}  \left\{ v \in V \ \middle| \  Q(v) = D , \red(\Delta(v)) \subset \cT_j^+[n]  \right\}, \ u \mapstoo u/p^n
	\]
	is bijective. We conclude the proof by recalling that the right hand side gives a set of representatives of 
	\[
	\left\{ v \in \gamma^\Z \backslash V \ \middle| \  Q(v) = D, \ \mathrm{depth}(v) \leq n \right\}.
	\]
\end{proof}

As a consequence, we obtain the following expression for $j_{\sD_\Phi(D)_e}\cdot j_{\sD_\Phi(D)_o}(\gamma)$.

\begin{theorem}\label{thm: expression j DPhi D in terms of Lj}
	Consider the same notation as above. 
	\begin{enumerate}
		\item If $\ord_p(D) = 0$, we have 
		\[
		j_{\sD_\Phi(D)}(\gamma) = \lim_{n \to +\infty} \prod_{j = 0}^{2t - 1} \frac{\prod_{v \in L_j^+[n]}\langle w_0^+ , v\rangle^{\Phi(v)}}{\prod_{v \in L_j^-[n]}\langle w_0^- , v\rangle^{\Phi(v)}}.
		\]
		\item If $\ord_p(D) = 1$, we have
		\[
		j_{\sD_\Phi(D)_e}\cdot j_{\sD_\Phi(D)_o}(\gamma) = \lim_{n \to +\infty} \prod_{j = 0}^{2t - 1}\frac{\prod_{v \in L_j^+[n] \cup L_j^+[n+1]}\langle w_0^+ , v\rangle^{\Phi(v)}}{\prod_{v \in L_j^-[n] \cup L_j^-[n+1]}\langle w_0^- , v\rangle^{\Phi(v)} }.
		\]	
	\end{enumerate}
\end{theorem}
\begin{proof}
	Suppose that $\ord_p(D) = 0$. By \eqref{eq: theta gamma vertices with depth} and Proposition \ref{prop: elementary but key}, we have 
	\[
	j_{\sD_\Phi(D)}(\gamma) = \lim_{n \to +\infty} \prod_{j = 0}^{2t - 1} \frac{\prod_{v \in L_j^+[n]}\langle w_0^+ , vp^{-n}\rangle^{\Phi(v)}}{\prod_{v \in L_j^-[n]}\langle w_0^- , v p^{-n}\rangle^{\Phi(v)}}.
	\]
	Here we used that $w_0^+$ (resp.~$w_0^-$) generates the line $\xi^+$ (resp.~$\xi^-$) and that $\Phi(pv) = \Phi(v)$ for every $v \in V$. Since the divisor $\sD_\Phi(D)$ is of degree $0$, the product of the factors $p^{-n\Phi(v)}$ is equal to $1$, leading to the desired expression. The case when $\ord_p(D) = 1$ is proven in an analogous way, but using \eqref{eq: theta gamma edges with depth}, instead of \eqref{eq: theta gamma vertices with depth}.
\end{proof}

\subsection{Computation of $p$-adic valuations} 
We end the section by using the previous calculations to compute the $p$-adic valuation of  $j_{\sD_\Phi(D)_e} \cdot j_{\sD_\Phi(D)_o}(\gamma)$.
\begin{proposition}\label{prop: j(D) is a p adic unit}
Let $\Phi$ be a convenient Schwartz--Bruhat function. Then
\[
\ord_p(j_{\sD_\Phi(D)_e}\cdot j_{\sD_\Phi(D)_o}(\gamma)) = 0
\]
for all $\gamma\in \Gamma$.
In particular, we have
\[
G^+_\Phi(q)\in \Hom(\Gamma,\Z_p^\times)[[q]].
\]
\end{proposition}
\begin{proof}
	Using Theorem \ref{thm: expression j DPhi D in terms of Lj}, together with the fact that $\ord_p(\langle w_0^+ , v\rangle ) = -j$ if $v \in L_j^+[n]$, and $\ord_p(\langle w_0^- , v \rangle ) = j$ if $v \in L_j^-[n]$, we deduce that it is enough to show that, for every $j \in \{ 0, \dots, 2t - 1\}$ and for every $n \geq 0$, we have 
	\[
	\sum_{v \in L_j^+[n]} \Phi(v) + \sum_{v \in L_j^+[n+1]}\Phi(v) = 0
	\]
	and that the same statement replacing the symbol $+$ with $-$ everywhere holds (which is proven analogously). Note that the quantity on the left hand side can be interpreted as follows. Denote by $\cA_j^+[n]$ the preimage of $\cT_j^+[n]$ under the reduction map. Then,
	\[
	\hat{\sD}_{\Phi}(D) \cap \left( \cA_j^+[n] \cup \cA_j^+[n+1]\right) = \sum_{\substack{v \in V \\ Q(v) = D \\ \red(\Delta(v)) \subset \cT_j^+[n]}} \Phi(v) \Delta(v) + \sum_{\substack{v \in V \\ Q(v) = D \\ \red(\Delta(v)) \subset \cT_j^+[n+1]}} \Phi(v) \Delta(v),
	\]
	and 
	\begin{equation*}
	\begin{split}
	\deg\left(\hat{\sD}_{\Phi}(D) \cap \left( \cA_j^+[n] \cup \cA_j^+[n+1]\right) \right) & = \sum_{\substack{v \in V \\ Q(v) = D \\ \red(\Delta(v)) \subset \cT_j^+[n]}} 2\Phi(v) + \sum_{\substack{v \in V \\ Q(v) = D \\ \red(\Delta(v)) \subset \cT_j^+[n+1]}} 2\Phi(v) \\ & = \sum_{v \in L_j^+[n]} 2\Phi(v) + \sum_{v \in L_j^+[n+1]} 2\Phi(v),
	\end{split}
	\end{equation*}
	where in the last equality we used Lemma \ref{lemma: elementary but key with details} together with the fact that $\Phi$ is invariant under multiplication by $p$. Recall that for a given lattice $L$ we defined the wide open $\cW_L \subset \sH_p$ in Section \ref{subsec: div of strong degree 0}. Since we have the disjoint union
	\[
	\cA_j^+[n] \cup \cA_j^+[n+1] = \bigcup_{L \in \cT_j^+[n+1] \cap \cT_0}   \cW_{L} \cup \bigcup_{L \in \cT_j^+[n] \cap \cT_0} \cA_L
	\]
	and $\Delta_\Phi(D)$ is of strong degree $0$, the result follows.
\end{proof}

\section{First order $p$-adic deformations of ternary theta series}\label{section: p-adic deformations of theta series}
Fix a convenient Schwartz--Bruhat function $\Phi$.
By Proposition \ref{prop: j(D) is a p adic unit} above we know that $G_\Phi^+(q)$ belongs to $\Hom(\Gamma, \Z_p^\times)[[q]]$.
In order to prove modularity of $G_\Phi^+(q) \in \Hom(\Gamma , \Z_p^\times)_{\Q}[[q]]$ it is therefore enough to prove that
\[
\log_\gamma(G_\Phi^+)(q)\defeq \sum_{D \in \cD_S} \log_p(j_{\sD_\Phi(D)_e}\cdot j_{\sD_\Phi(D)_o}(\gamma)) q^{D} \in \Q_p[[q]]
\]
is a modular form for every $\gamma \in \Gamma$ hyperbolic at $p$, which we fix from now on.
Here $\log_p$ denotes the branch of the $p$-adic logarithm such that $\log_p(p)=0$.
In this section, we use $\gamma$ and $\Phi$ to construct a $p$-adic family of theta series $\Theta_k$, of weight $k + 3/2$ and level $\Gamma_0(4N)$, satisfying the following two properties. First, $\Theta_0 = 0$. Second, if we denote by $\Theta_0'$ the derivative with respect to the $p$-adic variable $k$ evaluated at $k = 0$, and $e_{\mathrm{ord}}$ the so-called $p$-ordinary projector, then $e_{\ord}(\Theta_0') \in S_{3/2}(\Gamma_0(4N), \Q_p)$.
Furthermore, the generating series $\log_\gamma(G_\Phi^+)(q)$ is the projection to the  $U_{p^2} = 1$ eigenspace of $2e_{\mathrm{ord}}\Theta_0 '$.
In particular, it is a cusp form of weight $3/2$ and level $\Gamma_0(4N)$, which proves Theorem \ref{thm:main p-adic}.

\subsection{Shimura correspondence and ordinary subspaces}\label{subsec: Shimura correspondence} Let $k$ be a non-negative integer. For $\ell \nmid N$ denote by $T_{\ell^2}$ (resp.~$T_\ell$) the associated Hecke operator acting on $S_{k + 3/2}(\Gamma_0(4N), \Z)$ (resp.~$S_{2k + 2}(\Gamma_0(4N), \Z)$ and if $\ell \mid  N$ denote by $U_{\ell^2}$  (resp.~$U_\ell$) the associated Hecke operator acting on $S_{k + 3/2}(\Gamma_0(4N), \Z)$ (resp.~$S_{2k + 2}(\Gamma_0(4N), \Z)$. Let $D$ be a square-free integer such that $(-1)^{k+1} D > 0$. Denote by 
\[
\sS_{k,D}\colon S_{k+3/2}(\Gamma_0(4N), \Q) \too S_{2k+2}(\Gamma_0(2N),\Q)
\]
the Shimura lifting map associated to $D$ and $N$. It is given by the following formula 
\[
\sum_{n \geq 1} a_n q^n \mapstoo \sum_{n \geq 1} \left( \sum_{\substack{ d \mid  n \\ (d, N) = 1}}\left( \frac{D}{d} \right)d^{k}a_{\lvert D \rvert n^2/d^2} \right) q^n.
\]
This map is equivariant with respect to the Hecke operators introduced above. We will need the following key theorem. 

\begin{theorem}\label{thm: Shimura iso}
		Let $k \geq 0$. There exists a finite collection $\{D_j\}_j$ such that $D_j \equiv 0$ modulo $8N$ and $\{c_j\}_j \in \Q$ such that 
		\[
		\sS_k = \sum_j c_j \sS_{k,D_j}\colon S_{k+3/2}(\Gamma_0(4N),\Q) \xlongrightarrow{\sim} S_{2k+2}(\Gamma_0(2N), \Q)
		\]
		is a Hecke-equivariant isomorphism.
\end{theorem}
\begin{proof}
	The result follows from Remark 1 on Page 221 of \cite{MRV} and the fact that the $\C$-span of $S_{k+3/2}(\Gamma_0(4N), \Q)$ is equal to the space of cusp forms of weight $3/2$ and level $\Gamma_0(4N)$ by the theorem of Serre and Stark.
\end{proof}

Consider the space $\Z_p[[q]] \otimes_{\Z_p} \Q_p$ equipped with the norm
\[
\Big\lvert \sum_{n \geq 0} a_n q^n \Big\rvert = \mathrm{max}_n \{ \lvert a_n \rvert \}.
\]
Since the eigenvalues of $U_{p^2}$ acting on $S_{k+3/2}(\Gamma_0(4N), \Q)$ are algebraic integers, the operator 
\[
e_{\ord}\colon S_{k+3/2}(\Gamma_0(4N), \Z_p) \too S_{k + 3/2}(\Gamma_0(4N), \Z_p), \ f \mapstoo \lim_{m\to +\infty} U_{p^2}^{m!}(f)
\]
is well-defined.
Denote by $S_{k+3/2}^{\ord}(\Gamma_0(4N), \Z_p)$ the image of this map, and similarly define $S_{k+3/2}^{\ord}(\Gamma_0(4N), \Q_p)$.
We also consider the analogous definition for integral weight cusp forms and use similar notation. 

\begin{proposition}\label{prop: dim ordinary spaces for half-integral weight constant}
	The rank of the finitely generated modules $S_{k + 3/2}^{\ord}(\Gamma_0(4N), \Z_p)$
	is constant as long as $k$ varies over non-negative integers such that $k \equiv 0 \bmod (p-1)/2$.
\end{proposition}
\begin{proof}
	It is enough to prove that $\mathrm{dim}_{\Q_p} S_{k + 3/2}^{\ord}(\Gamma_0(4N), \Q_p)$ is constant as long as $k \in \Z_{\geq 0}$ and $k \equiv 0 \bmod(p-1)/2$. Viewing $S_{k + 3/2}(\Gamma_0(4N), \Q_p)$ and $S_{2k + 2}(\Gamma_0(2N), \Q_p)$ as subsets of $\Z_p[[q]] \otimes \Q_p$, we see that the extension of scalars to $\Q_p$ of the Shimura isomorphism given in Theorem \ref{thm: Shimura iso} is Hecke-equivariant and continuous. It follows that $e_{\ord}\sS_k = \sS_k e_{\ord}$. From there, we deduce that the restriction of $\sS_k$
	\[
	S_{k + 3/2}^{\ord}(\Gamma_0(4N), \Q_p) \xlongrightarrow{\sim} S_{2k + 2}^{\ord}(\Gamma_0(4N), \Q_p).
	\]
	is an isomorphism.	Since the dimensions of the right hand side are constant as long as $k \equiv 0 \bmod (p-1)$ (see proof of Theorem 3 in Section 7.2 of \cite{H1}), the result follows. 
\end{proof}

\subsection{$\Lambda$-adic forms of half-integral weight}\label{subsec: Lambda adic forms of half-integral weight}
We study the space of $\Lambda$-adic modular forms of half-integral weight and prove a classicality result in this setting. We follow \cite{H1} and \cite{H2}.

Let $\Lambda = \Z_p[[T]]$ denote the Iwasawa algebra over $\Z_p$ and put $u=1+p \in 1 + p\Z_p$. 
A {\em $\Lambda$-adic cusp form of half-integral weight} is a formal power series
\[
F = \sum_{n \geq 1} A_n q^n \in \Lambda[[q]]
\]
such that there exists $k_0$ (dependent on $F$) satisfying that for all $k \geq k_0$ and $k \equiv 0 \bmod (p-1)$, the so-called weight $k$ specialization
\[
F_k\defeq F(u^k - 1) \defeq \sum_{n \geq 1} A_n(u^k - 1)q^n \in \Z_p[[q]],
\]
belongs to $S_{k + 3/2}(\Gamma_0(4N), \Z_p)$.
We denote the space of such forms by $\PP$. We define {\em ordinary} $\Lambda$-adic cusp forms of half-integral weight in the same way as above but replacing $S_{k + 3/2}(\Gamma_0(4N), \Z_p)$ by $S_{k + 3/2}^{\ord}(\Gamma_0(4N), \Z_p)$, and we denote this space by $\PP^{\ord}$. 

A key input to study the space $\PP^{\ord}$ is the fact that $r^{\rm{ord}} = \mathrm{rank}_{\Z_p}S_{k+3/2}^{\ord}(\Gamma_0(4N), \Z_p)$ is constant as long as $k \geq 0$ and $k \equiv 0 \bmod (p-1)$, proven in Proposition \ref{prop: dim ordinary spaces for half-integral weight constant}. 
\begin{theorem}\label{PPord is free of finite rank over Lambda}
	$\PP^{\ord}$ is free of finite rank over $\Lambda$. In particular, $\mathrm{rank}_\Lambda(\PP^{\ord}) \leq r^{\ord}$.
\end{theorem}
\begin{proof}
	A proof of this statement can be found in Proposition 4 of \cite{H2}. There, Hida considers different level structures than the ones considered here, but the same reasoning works in this case. 
\end{proof}

For every $k \geq 0$, we can define a map 
\[
\varphi_k\colon \PP^{\ord}/P_k \PP^{\ord}\too \Z_p[[q]], \ F \mapstoo F_k,
\]
where $P_k = T - (u^k - 1) \in \Z_p[[T]]$, which is injective. The image of this map is a submodule of $\Z_p[[q]]$. We can also view $S_{k+3/2}^{\ord}(\Gamma_0(4N), \Z_p)$ as a submodule of $\Z_p[[q]]$. The relation between these two submodules is the so-called control theorem, which is again a consequence of Proposition \ref{prop: dim ordinary spaces for half-integral weight constant}. 

\begin{theorem}\label{thm: classicality for lambda adic forms}
	Let $k \geq 0$ such that $k \equiv 0 \bmod (p-1)$. Then, the map $\varphi_k$ induces an isomorphism
	\[
	\varphi_k\colon \PP^{\ord}/P_k \PP^{\ord} \xlongrightarrow{\sim} S_{k+3/2}^{\ord}(\Gamma_0(4N), \Z_p).
	\]
\end{theorem}
\begin{proof}
	The analogous statement for ordinary cuspidal $\Lambda$-adic forms of integral weight is known. A proof can be found in Theorem 3, Section 7.3 of \cite{H1}. The same proof given there works for the case of half-integral weight forms once we have Proposition \ref{prop: dim ordinary spaces for half-integral weight constant}.
	
	Indeed, it can be proven that every element $f \in S_{k+3/2}^{\ord}(\Gamma_0(4N), \Z_p)$ is in the image $\varphi_k$ as in the case of integral weight forms. For example, this is done in Proposition 5 of \cite{H2}. Since $\PP^{\ord}$ is free of finite rank, this already implies the result for $k$ large enough. To obtain the result for general all $k$ note that 
	\begin{equation}\label{eq: chain of inclusions control thm}
		S_{k+3/2}^{\ord}(\Gamma_0(4N), \Z_p) \subset \mathrm{Im}(\varphi_k) \subset \Z_p[[q]].
	\end{equation}
	This implies
	\[
	r^{\ord} \leq \mathrm{rank}_{\Z_p}(\mathrm{Im}(\varphi_k)).
	\]
	Since $\mathrm{rank}_{\Z_p}(\mathrm{Im}(\varphi_k)) \leq \mathrm{rank}_\Lambda(\PP^{\ord})$, the previous inequality and Proposition \ref{PPord is free of finite rank over Lambda} imply $r^{\ord} = \mathrm{rank}_{\Z_p}(\mathrm{Im}(\varphi_k))$. Hence, it follows from \eqref{eq: chain of inclusions control thm} that $\mathrm{Im}(\varphi_k) = 	S_{k+3/2}^{\ord}(\Gamma_0(4N), \Z_p) $ and we are done.
\end{proof}
Fix a $\Lambda$-basis $\{B_1, \dots, B_r\}$ of $\PP^{\ord}$ and write 
\[
B_i = \sum_{n \geq 1} A_{i,n} q^n \in \Lambda[[q]].
\]
By Theorem \ref{thm: classicality for lambda adic forms}, the set $\{B_1(0), \dots, B_r(0)\}$ forms a $\Z_p$-basis of $S_{3/2}^{\rm ord}(\Gamma_0(4N), \Z_p)$.
Thus, there exist $n_1, \dots, n_r$ such that 
\[
\det \left(   (A_{i, n_j}(0))_{1 \leq i , j \leq r}   \right) \neq 0. 
\]
Since $\det((A_{i, n_j}(T))_{1 \leq i , j \leq r}  ) \in \Lambda$, it follows by continuity that there exists $k_0$ such that if $k \geq k_0$ and $k \equiv 0 \bmod (p-1)$, then 
\begin{equation}\label{eq: det nonzero for k large enough}
	\det \left(   (A_{i, n_j}(u^k - 1))_{1 \leq i , j \leq r}   \right) \neq 0. 
\end{equation}
Now define 
\[
\underline{b_i} = \sum_{n \geq 1} a_{i, n }q^n,
\]
where $a_{i,n}\colon \Z_p \to \Z_p$ is the analytic function determined by $a_{i,n}(k) = A_{i, n}(u^k - 1)$ for every $k \geq 0$ such that $k \equiv 0 \bmod (p-1)$. 

We will now prove that certain first order derivatives of $\Lambda$-adic modular forms of half-integral weight are modular forms themselves.
Let $F$ be a $\Lambda$-adic modular form of half-integral weight such that $F_0=0$.
Let 
\[ 
F' \defeq \frac{d}{dk}\left.F_k\right|_{k=0} = \lim_{k \to 0} \frac{F_k}{k} \in \Z_p[[q]]
\]
be the first derivative of $F$ with respect to $k$ evaluated at $k = 0$.
It is a weight $3/2$ analogue of a $p$-adic modular form in the sense of Serre.
Here the limit is taken in $\Z_p[[q]] \otimes \Q_p$ with respect to the norm introduced above.
Recall that $U_{p^2}$ has the following expression at the level of $q$-expansions:
\[
\sum_{n \geq 0} a_nq^n \mapstoo \sum_{n \geq 1} a_{np^2}q^n.
\]
Since $\lvert U_p^2 f \rvert \leq \lvert f \rvert$ for any $f \in \Z_p[[q]] \otimes \Q_p$ and $U_{p^2}$ is linear, it follows that we can define the $p$-adic modular form of weight $3/2$
\[
e_{ \ord}(F') \defeq \lim_{k \to 0}  e_{\mathrm{ord}} \left(\frac{F_k}{k}\right).
\]
Moreover, it is a calculation to verify that the limit 
\[
\lim_{m \to +\infty} U_{p^2}^{m!} (F')
\]
exists in $\Z_p[[q]] \otimes \Q_p$  and is equal to $e_{\ord}(F')$.

\begin{corollary}\label{cor: classicality}
For every $F\in\PP$ with $F_0=0$ the $p$-adic modular form $e_{\ord}(F_0')$ is classical. More precisely, it belongs to $S_{3/2}^{\ord}(\Gamma_0(4N), \Q_p)$.
\end{corollary}
\begin{proof}
By definition, 
	\[
	e_{\ord}(F') = \lim_{k \to 0}e_{\ord}\left( \frac{F_k}{k} \right).
	\]
	Now, Theorem \ref{thm: classicality for lambda adic forms} implies that, for every $k > 0$ and $k \equiv 0 \bmod (p-1)$, we can write 
	\[
	e_{\ord}\left( \frac{F_k}{k} \right) = \sum_{i = 1}^r x_i(k) \underline{b_i}(k),
	\]
	where $x_i(k) \in \Q_p$ for every $i$. Let $n_1, \dots, n_r$ be as above, and note that $(x_i(k))_i$ is the solution of the linear system of equations
	\[
	\left(a_{i, n_j}(k)\right)_{j,i}(x_i(k))_i = \left( a_{n_j}\left(\frac{F_k}{k} \right) \right)_j.
	\]
	Moreover, since the determinant of the matrix defining this system is an analytic function, which is non-zero if $k \geq k_0$ and $k \equiv 0 \bmod (p-1)$ by \eqref{eq: det nonzero for k large enough} and the discussion above it, we deduce that for every $i$ the limit $\lim_{k \to 0} x_i(k)$ exists in $\Q_p$.
	Denote it by $x_i(0)$. Then, 
	\[
	e_{\ord}(F') = \lim_{k \to +\infty } e_{\ord}\left( \frac{F}{k} \right) = \sum_{i = 1}^r x_i(0) \underline{b_i}(0)
	\]
	and it follows from Theorem \ref{thm: classicality for lambda adic forms} in the particular case that $k = 0$ that the right hand side belongs to $S_{k + 3/2}^{\ord}(\Gamma_0(4N), \Q_p)$, which concludes the proof.
\end{proof}

\subsection{$p$-adic families of theta series}\label{subsec: p adic family Thetak} Recall that the element $\gamma \in \Gamma$ determines a collection of $\Z_p$-lattices $L_j$ of depth zero, and let $w_j^+$ and $w_j^-$ be generators of the $\Z_p$-modules $L_j[\varpi]$ and $L_j[\varpi^{-1}]$ respectively. Note that $w_j^+ , w_j^-$ can be viewed both as elements of $V_{\Q(\gamma)}$ and $V_{\Q_p}$, using the embedding $\Q(\gamma) \into \Q_p$ satisfying that $\ord_p(\varpi) = 2t > 0$. These data, together with $\Phi$, can be used to define the following Schwartz--Bruhat functions
\[
\Phi_j^+ = \Phi \otimes 1_{\left\{v \in L_j \ \middle| \  \langle v, w_j^+ \rangle \in \Z_p^\times\right\}}, \mbox{ and } \Phi_j^- = \Phi \otimes 1_{\left\{v \in L_j \ \middle| \  \langle v, w_j^- \rangle \in \Z_p^\times\right\}}
\]
on $V_{\A^\infty}$ for every $j \in \{ 0, \dots , 2t -1 \}$. We have that $\Phi$ is invariant under $K_0(4N/p)^{(p)}$ by assumption. Moreover,
\begin{equation}\label{eq: expression Schwartz--Bruhat at p telescope}
	1_{\left\{v \in L_j \ \middle| \  \langle v, w_j^+ \rangle \in \Z_p^\times\right\}} = 1_{L_j} - 1_{L_j \cap L_{j-1}}, \
	1_{\left\{v \in L_j \ \middle| \  \langle v, w_j^- \rangle \in \Z_p^\times\right\}} = 1_{L_j} - 1_{L_j \cap L_{j+1}}
\end{equation}
and $L_j$ is unimodular, while $L_j \cap L_{j-1}$ has level $p$ for every $j$. It follows that $\Phi_j^{\pm}$ is invariant under $K_0(4N)$ for every $j$. Since $w_j^+$ and $w_j^-$ are isotropic, the functions $v\mapsto \langle w_j^\pm,v\rangle^k$ are harmonic polynomials on $V_{\Q(\gamma)}$ for all integers $k \geq 0$. Hence, the $q$-series
\begin{equation}\label{eq: p-adic family Thetak}
\Theta_k \defeq \sum_{j = 0}^{2t - 1}\sum_{v\in V} \Phi_j^+(v) \langle w_j^+, v\rangle^k q^{Q(v)} - 
\sum_{j = 0}^{2t - 1} \sum_{v\in V} \Phi_j^-(v) \langle w_j^-, v\rangle^k q^{Q(v)}
\end{equation}
is a linear combination of classical theta-series with coefficients in the quadratic imaginary field $\Q(\gamma)$ of weight $k+3/2$ and level $\Gamma_0(4N)$ by \cite[Theorem 4.1]{BoAutomorphic}.
Via the embedding $\Q(\gamma) \into \Q_p$, we can also view the Fourier coefficients of $\Theta_k$ as elements in $\Z_p$. Moreover, since the non-zero terms in the infinite sum defining $\Theta_{k}$ solely involve elements of $V$ for which $\langle w_j^+,v\rangle$ (resp.~$\langle w_j^-,v\rangle$) are $p$-adic units, it follows that the Fourier coefficients of $\Theta_{k}$ vary analytically  as functions of the variable $k\in (\Z/(p-1)\Z)\times \Z_p$. We can therefore define $\Theta_k$ for $k \in (\Z/(p-1)\Z) \times \Z_p$. It gives a prototypical instance of a $\Lambda$-adic modular form of half-integral weight, in the sense that there exists a $F \in \PP$ such that $F_k = \Theta_k$ for every $k\equiv 0 \bmod (p-1)$.

\begin{lemma}\label{lem: vanishing}
	The weight $3/2$ specialization $\Theta_{0}$ is identically zero. 
\end{lemma}
\begin{proof}
	By \eqref{eq: expression Schwartz--Bruhat at p telescope}, we have
	\begin{equation*}
	\begin{split}
	\Theta_0 & =  \sum_{j = 0}^{2t - 1}\sum_{v \in V}\Phi(v) \left( 1_{L_j} - 1_{L_j \cap L_{j-1}}\right)(v) q^{Q(v)} - \sum_{j = 0}^{2t - 1}\sum_{v \in V}\Phi(v) \left( 1_{L_j} - 1_{L_j \cap L_{j+1}}\right)(v) q^{Q(v)} \\ & = -\sum_{v \in V} \Phi(v)1_{L_0 \cap L_{-1}}(v)q^{Q(v)} + \sum_{v \in V} \Phi(v) 1_{L_{2t -1} \cap L_{2t}}(v) q^{Q(v)} = 0,
	\end{split}
	\end{equation*}
	where in the last equality we used that $\gamma(L_0 \cap L_{-1}) = L_{2t} \cap L_{2t - 1}$ and that the functions $\Phi$ and $v \mapsto Q(v)$ are invariant under the action of $\gamma^\Z$. 
\end{proof}

Lemma \ref{lem: vanishing} together with Corollary \ref{cor: classicality} immediately imply the following:
\begin{corollary}\label{prop: classicality eordTheta0'}
	The $p$-adic modular form $e_{\ord}(\Theta_0')$ is classical. More precisely, it belongs to $S_{3/2}^{\ord}(\Gamma_0(4N), \Q_p)$.
\end{corollary}

We now relate $e_{\mathrm{ord}}(\Theta_0')$ with the generating series $\log_\gamma(G_\Phi^+)(q)$.

\begin{lemma}\label{thm: computation q-expansion ordinary projection}
	For every $D \in \cD_S$ and every $n \geq 0$ the following equality holds:
	\[
	a_{Dp^{2n}}\left( \log_\gamma(G_\Phi^+)(q))\right) = a_{Dp^{2n}}\left( e_{\ord}\left( (1 + U_{p^2})\Theta_{0}' \right)   \right).
	\]
	In particular, the right hand side does not depend on $n$.
\end{lemma}
\begin{proof}
	It is enough to prove the formula when $\ord_p(D) \in \{ 0, 1\}$. Using \eqref{eq: p-adic family Thetak},  we can compute
	\[
	a_{Dp^{2m}}(\Theta_0') = \sum_{j = 0}^{2t - 1}\sum_{v \in L_j^+[m]} \Phi(v)\log_p( \langle w_j^+ , v \rangle)  - \sum_{j = 0}^{2t - 1} \sum_{v \in L_j^-[m]} \Phi(v) \log_p( \langle w_j^- , v\rangle ).
	\]
	Hence, it follows from Theorem \ref{thm: expression j DPhi D in terms of Lj} that 
	\begin{equation}\label{eq: a_D log G in terms of a limit of Theta_0'}
	a_{D}\left( \log_\gamma(G_\Phi^+)(q))\right) = \lim_{m \to +\infty} a_{Dp^{2m}}(\Theta_0') + a_{Dp^{2(m+1)}}(\Theta_0').
	\end{equation}
	Note that it is also a consequence of Theorem \ref{thm: expression j DPhi D in terms of Lj}, that the limit on the right hand side exists. 
	
	On the other hand, from the expression of the ordinary projection given above, we have that for every $n \geq 0$
	\begin{equation*}\label{eq: formula ordinary projector}
		a_{Dp^{2n}}(e_{\ord}(\Theta_0' + U_{p^2}(\Theta_0'))) = \lim_{m \to \infty} a_{D p^{2(n + m!)}}(\Theta_0') + a_{D p^{2(n + m! + 1) }}(\Theta_0') .
	\end{equation*}
	Since the right hand side of the previous equation is a subsequence of the right hand side of \eqref{eq: a_D log G in terms of a limit of Theta_0'}, we deduce that 
	\[
	a_{D}\left( \log_\gamma(G_\Phi^+)(q))\right) = a_{Dp^{2n}}(e_{\ord}(\Theta_0' + U_{p^2}(\Theta_0'))).
	\]
\end{proof}

The action of $U_{p^2}$ on $S_{3/2}^{\ord}(\Gamma_0(4N), \overline{\Q}_p)$ diagonalizes. This can be justified, for example, using Theorem \ref{thm: Shimura iso} and the fact that the analogous statement for weight $2$ forms of level $\Gamma_0(2N)$ is well-known. In particular, we can consider
\[
\mathrm{pr}_{1}\colon S_{3/2}^{\ord}(\Gamma_0(4N), \Q_p) \too S_{3/2}^{\ord}(\Gamma_0(4N), \Q_p)
\]
to be the projection to the $U_{p^2} = 1$ eigenspace. We prove the main identity of this work, which implies Theorem \ref{thm:main p-adic}.

\begin{theorem}\label{thm: log(GPhi(gamma)) = pr1(eordTheta_0')}
	The following identity of formal power series holds:
	\[
	\log_\gamma(G_\Phi^+)(q) = 2\mathrm{pr}_{1}(e_{\ord}( \Theta_{0}')).
	\]
	In particular, $\log_\gamma(G_\Phi^+)(q)$ is an element of $S_{3/2}^{\ord}(\Gamma_0(4N), \Q_p)$.
\end{theorem}
\begin{proof}
	Since $U_{p^2}$ and $e_{\mathrm{ord}}$ commute, it is enough to prove that if 
	\[
	f = e_{\ord}( \Theta_{0}' + U_{p^2} \Theta_0') = \sum_{n \geq 1} a_n(f) q^n,
	\]
	we have $\log_\gamma(G_\Phi^+)(q) = \mathrm{pr}_1(f)$. Note that by Theorem \ref{thm: computation q-expansion ordinary projection}, we have that if $D \in \cD_S$,
	\[
	a_{Dp^{2n}}(\log_\gamma(G_\Phi^+)(q)) = a_{Dp^{2n}}(f).
	\]
	for every $n \geq 0$. Therefore, the equality of the theorem follows from proving that, if $D \in \Z_{\geq 0}$ is such that $\ord_p(D) \in \{ 0 , 1\}$ and $N \geq 0$:
	\begin{enumerate}
		\item\label{(1)} If $\left( \frac{-D}{p} \right)= 1$, $a_{Dp^{2n}}(\mathrm{pr}_{1}(f)) = 0$.
		\item\label{(2)} If $\left( \frac{-D}{p} \right) \in \{ 0, -1\}$, $a_{Dp^{2n}}(\mathrm{pr}_{1}(f)) = a_{Dp^{2n}}(f)$.
	\end{enumerate}
	We begin proving the first point.The Atkin--Lehner involution at $p$, denoted $w_p$, acts by multiplication with $-1$ on $\mathrm{pr}_{1}(f)$.
	Then, \eqref{(1)} follows from the description of the $-1$ eigenspace for $w_p$ given in Remark 2 of \cite{MRV}.
	We proceed to prove the second point. Write $f$ as a sum of eigenvectors for $U_{p^2}$, namely
	\[
	f = \sum_{i = 1}^r f_i, 
	\]
	where $f_i \in S_{3/2}^{\ord}(\Gamma_0(4N), L)$ and there exists $\alpha_i$ such that $U_{p^2}f_i = \alpha_i f_i$ for every $i$. Here $L$ is a finite extension of $\Q_p$ containing all the elements $\alpha_i$. We can suppose without loss of generality that $\alpha_i \neq \alpha_j$ if $i \neq j$ and that $\alpha_1 = 1$. In particular, $f_1 = \mathrm{pr}_{1}(f)$ (which is possibly zero). Let $D$ be such that it satisfies the conditions of (2). For every $n \geq 0$, we can consider the $Dp^{2n}$-th Fourier coefficient of each side of the previous equality to obtain 
	\[
	a_D(f) = a_D(f_1) + \sum_{i = 2}^r \alpha_i^n a_D(f_i),
	\] 
	where we used that $a_D(f) = a_{Dp^{2n}}(f)$ for every $n \geq 0$, which holds by Theorem \ref{thm: computation q-expansion ordinary projection}. Considering this equality for $n = 0, \dots, r-1$ and using that the Vandermonde matrix associated to $\{ 1, \alpha_2, \dots, \alpha_r\}$ is non-singular we deduce that we must have $a_D(f) = a_D(f_1)$, implying the desired equality. Once we have $\log_\gamma(G_\Phi^+)(q) = 2\mathrm{pr}_1(e_{\ord}( \Theta_{0}'))$, the fact that $\log_\gamma(G_\Phi^+)(q) \in S_{3/2}(\Gamma_0(4N), \Q_p)$ follows from Corollary \ref{prop: classicality eordTheta0'}.
\end{proof}

\section{Numerical example}\label{section: numerical example}

We conclude by presenting a concrete example where we numerically compute the $p$-adic family $\Theta_k$ and the reduction modulo $p$ of $e_{\mathrm{ord}}(\Theta_0'/p)$.

Let $S = \{7, 13, \infty\}$, let $p = 7$ and consider $B$ be the quaternion algebra over $\Q$ ramified exactly at $\{13, \infty \}$. It can be viewed as the algebra over $\Q$ generated by $i, j, k$ where 
\[
i^2 = -2 , \ j^2 = -13, \ ij = -ji = k.
\] 
Let $\tilde{R}$ be the maximal $\Z[1/p]$-order of $B$ given by $\left\langle {1}/{2} + {j}/{2} + {k}/{2}, {i}/{4}  + {j}/{2}  + {k}/{4}, j, k \right\rangle$, let $\alpha = 1 + i \in B^\times$, which has reduced norm $\ell = 3 \not \in S$, and consider the Eichler $\Z[1/p]$-order $R = \tilde{R} \cap \alpha \tilde{R}\alpha^{-1}$ of level $3$. Denote by $\Gamma$ the group of norm one units in $R$. The quotient $\Gamma \backslash \sH_p$ is isomorphic to the $\C_p$-points of the Shimura curve $X$.

\subsection{Construction of the $p$-adic family $\Theta_k$}

Recall the definition of the $p$-adic family $\Theta_k$ given in \eqref{eq: p-adic family Thetak}. This family depends on a choice of a Schwartz--Bruhat function $\Phi$, an element $\gamma \in \Gamma$ hyperbolic at $p$ and the eigenvectors of the action of $\gamma$ on $V_{\Q_p}$. We proceed to fix these data.
Let $\tilde{R}_0$ be the subgroup of elements of $\tilde{R}$ of reduced norm zero.
As before, write $1_{\tilde{R}_0}$ for the characteristic function of $\tilde{R}_0\otimes \hat{\Z}^{(p)}$.
Consider the $\hat{R}^\times \times K_0(4\cdot 13)^{(p)}$-invariant Schwartz--Bruhat function 
\[
\Phi = 1_{\tilde{R}_0 } - 1_{\tilde{R}_0} \cdot \alpha^{-1}.
\]
Since we have the factorization of ideals $(7) = (7, x + 3)(7, x + 4)$ in the ring of integers of $\Q[x]/(x^2 + 5)$, the element $(x+3)/(-x+3) = 3x/7 + 2/7$ is a $p$-unit in $\Q[x]/(x^2 + 5)$. Its image in $B$ with respect to the embedding
\[
\Q[x]/(x^2 + 5) \intoo B, \ x \mapstoo \frac{i}{4} + \frac{j}{2} + \frac{k}{4},
\]
is equal to
\[
\gamma = \frac{2}{7} + \frac{3i}{28} + \frac{3j}{14} + \frac{3k}{28},
\]
and it can be verified that $\gamma \in \Gamma$.
Let $\mathfrak{p}$ be the prime ideal spanned by $7$ and $ x + 4$ and fix the embedding 
\begin{equation}\label{eq: embedding Q(sqrt-5) in Qp}
	\Q[x]/(x^2 + 5) \intoo \Q_p
\end{equation}
such that $\ord_p((\mathfrak p)) = 1$.
Using that $\gamma$ is hyperbolic at $p$, we deduce that its action on $V\otimes \Q[x]/(x^2 + 5)$ (and therefore on $V_{\Q_p}$) diagonalizes.
The eigenvectors of $\gamma$ are
\[
w^+ = i + \left( \frac{4x}{39} - \frac{2}{39}\right)j + \left( -\frac{4x}{39} - \frac{1}{39}\right)k
\]
\[
e = i + 2j + k
\]
\[
w^- = i + \left( -\frac{4x}{39} - \frac{2}{39}\right)j + \left( \frac{4x}{39} - \frac{1}{39}\right)k
\]
with eigenvalues $\varpi = -12x/49 - 41/49$, $1$ and $\varpi^{-1}$ respectively. Since $v_{\mathfrak p}(\varpi) = 2$, we have that $t = 1$. Note that $\langle w^+, w^- \rangle \in \Z_p^\times$, which implies that  $\{ w^+, w^- \}$ generate a hyperbolic plane. Finally, consider the unimodular $\Z_p$-lattices 
\[
L_0 = \langle w^+ , e, w^- \rangle = \langle i, j , k\rangle
\]
\[
L_1 = \langle pw^+ , e , w^-/p \rangle = \left\langle i + 2j + k, 14i -\frac{28j}{39} - \frac{14k}{39}, \frac{i}{7}  + j + \frac{8k}{7} \right\rangle.
\]

We can therefore consider the $p$-adic family $\Theta_k$ given in \eqref{eq: p-adic family Thetak} attached to the data $\Phi$, $\gamma$ and $\{ w^+ , e, w^- \}$.

\subsection{Calculation of $\Theta_0$ and $e_{\mathrm{ord}}(\Theta_0')$} Consider the same notation as above. For every $M \leq 421 \cdot p^2$ we can run over the following sets:
\begin{equation*}
	\begin{split}
		& \left\{v \in V \ \middle| \  \langle v , v \rangle = M,  1_{{\tilde{R}}_0}(v)\cdot 1_{L_0}(v) = 1 , \langle v , w^+\rangle \text { or } \langle v, w^- \rangle \in \Z_p^\times\right\}, \\
		& \left\{v \in V \ \middle| \  \langle v , v \rangle = M, 1_{{\tilde{R}}_0}(v)\cdot 1_{L_1}(v) = 1 , \langle v , pw^+\rangle \text { or } \langle v, p^{-1}w^- \rangle \in \Z_p^\times\right\}, \\
		& \left\{v \in V \ \middle| \  \langle v , v \rangle = M, 1_{\alpha \cdot {\tilde{R}_0}}(v)\cdot 1_{L_0}(v) = 1 , \langle v , w^+\rangle \text { or } \langle v, w^- \rangle \in \Z_p^\times\right\}, \\
		&\left\{v \in V \ \middle| \  \langle v , v \rangle = M, 1_{\alpha \cdot {\tilde{R}}_0}(v)\cdot 1_{L_1}(v) = 1 , \langle v , pw^+\rangle \text { or } \langle v, p^{-1}w^- \rangle \in \Z_p^\times\right\}.
	\end{split}
\end{equation*}
From there, it is possible to compute the first $421\cdot p^2$  Fourier coefficients of $\Theta_k$, for $k \in \Z$, as well as of $\Theta_0'$. In particular, define  
\[
\Theta_{{\tilde{R}_0}, j}^{+} := \sum_{\substack{v\in V \\  \langle v,p^{+ j}w^{+} \rangle \in \Z_p^\times}} 1_{\tilde{R}_0}(v) \cdot 1_{L_{j}}(v) q^{Q(v)}
\]
and define $\Theta_{{\tilde{R}_0}, j}^{-}$ with the same expression but replacing the symbol $+$ by the symbol $-$ everywhere. Define also $\Theta_{\alpha \cdot {\tilde{R}_0}, j}^{+}$ and $\Theta_{\alpha\cdot {\tilde{R}_0} , j}^{-}$ analogously. Then,
\[
\Theta_0 = \left( \Theta_{\tilde{R}_0, L_0}^+ + \Theta_{\tilde{R}_0, L_1}^+ - \Theta_{\tilde{R}_0, L_0}^- - \Theta_{\tilde{R}_0, L_1}^- \right) - 
\left( \Theta_{\alpha \cdot \tilde{R}_0, L_0}^+ + \Theta_{\alpha \cdot \tilde{R}_0, L_1}^+ - \Theta_{\alpha \cdot \tilde{R}_0, L_0}^- - \Theta_{\alpha \cdot \tilde{R}_0, L_1}^- \right)
\]
and we verify that the first $421\cdot p^2$ Fourier coefficients are $0$. For example, the first $4$ terms that appear in the previous expression are given below. 

\begin{table}[h]\label{table: q-expansions for Theta0}
	\begin{adjustwidth}{-1cm}{-1cm}
		\begin{center}
			\begin{tabular}{ ||c|c|c|c|c|c|c|c|c|c|c|c|c|c|c|c|c|c|| }
				\hline
				\multicolumn{1}{||c|}{Theta series} & \multicolumn{17}{c ||}{$q$-expansion} \\
				\multicolumn{1}{|| c | }{} & 
				2 & 5 & 6& 7& 8& 11& 13& 15& 18& 19& 20& 21& 24& 26 & 28 & 31 & 32 \\
				\hline
				
				$\Theta_{\tilde{R}_0, L_0}^+$  &2 & 2& 4& 4& 6& 8 & 2& 8 & 6& 6& 8& 8& 8& 6& 6 &10& 14 \\
				
				$\Theta_{\tilde{R}_0, L_0}^-$ & 2& 0& 4& 2& 6& 8& 0& 8& 6& 4& 8& 8& 12& 6& 6& 8 &14 \\
				
				$\Theta_{\tilde{R}_0, L_1}^+$ & 2& 0& 4& 2& 6& 8& 0& 8& 6& 4& 8& 8& 12& 6& 6& 8 &14 \\
				
				$\Theta_{\tilde{R}_0, L_1}^-$  &2 & 2& 4& 4& 6& 8 & 2& 8 & 6& 6& 8& 8& 8& 6& 6 &10& 14 \\
				\hline
			\end{tabular}
			\caption{ First Fourier coefficients of the theta series $\Theta_{\tilde{R}_0 , L_j}^{\pm}$. }
		\end{center}
	\end{adjustwidth}
\end{table}
The coefficients of $q^n$ for $n <32$ that do not appear in the table are $0$, as theta series attached to lattices in $V$ have non-zero Fourier coefficients only if $-D$ is not a square modulo $13$. We will follow a similar convention from now on. The forms on the previous table belong to $M_{3/2}(\Gamma_0(4\cdot 91), \Q)$, a space of dimension $32$, and these coefficients fully determine them. 

From \eqref{eq: p-adic family Thetak}, we see that the derivative of $\Theta_k$ with respect to $k$ evaluated at $k = 0$ is equal to 
\begin{equation}\label{eq: q-expansion Theta0'}
\Theta_0' = \sum_{j = 0}^{2t - 1} \sum_{v \in V} \Phi_j^+(v)\log_p\langle p^j w^+, v \rangle q^{Q(v)} - \sum_{j = 0}^{2t - 1} \sum_{v \in V} \Phi_j^-(v)\log_p\langle p^{-j} w^-, v \rangle q^{Q(v)}.
\end{equation}
Note that the dot products $\langle v, w^{\pm} \rangle$ belong to $\Q[x]/(x^2 + 5)$ and have $\mathfrak p $-adic valuation $0$. Using the embedding \eqref{eq: embedding Q(sqrt-5) in Qp}, we can view them as elements in $\Z_p^\times$. Therefore, the $p$-adic logarithm of these numbers lies in $p\Z_p$. We can then consider $\Theta_0'/p$ as an element in $\Z_p[[q]]$ and study its reduction modulo $p$.

Similarly as above, we can calculate the first $421\cdot p^2$ Fourier coefficients of $\Theta_0'/p$ modulo $p$. The first ones are
\[
\frac{\Theta_0'}{p} = 2q^2 + 3q^5 + 2q^6 + 4q^7 + 5q^8 + 4q^{11} + 3q^{13} + 3q^{15} + 2q^{18} + 3q^{20} + 6 q^{21} + \cdots.
\]
Since it is possible to calculate the first $421\cdot p^2$ Fourier coefficients of $\Theta_0'/p \bmod p$, we obtain the first $421$ Fourier coefficients of ${U}_{p^2}(\Theta_0'/p)$ modulo $p$. The first ones are 
\begin{equation*}
	\begin{split}
		{U}_{p^2}(\Theta_0'/p) & = 3q^2 + 3q^5 + 5q^6 + 2q^7 + 3q^{11} + 6q^{13} + 3q^{15} + 5q^{18} + q^{19} +  2q^{20} + 2q^{21}  \\ & +  2q^{24} + 3q^{26} + q^{28} + 6q^{31} + q^{32} + q^{33} + 4q^{34}+ q^{37} + 3q^{39} + q^{44} + \cdots.
	\end{split}
\end{equation*}
The following proposition, which is verified experimentally using the calculations mentioned above and Magma, is key for the next calculations.

\begin{proposition}\label{prop: Up^2 ofTheta0'/p is equal to a form of wt 3/2 mod p}
	There exists a cusp form in $S_{3/2}(\Gamma_0(4\cdot91), \Z)$ whose reduction modulo $p$ is equal to ${U}_{p^2}(\Theta_0'/p) \bmod p$.
\end{proposition}
\begin{proof}
	Since $\Theta_0 = 0$, we deduce from the expressions of  $\Theta_0'$ in \eqref{eq: q-expansion Theta0'} and of $\Theta_k$ in \eqref{eq: p-adic family Thetak} that 
	\[
	\frac{\Theta_0'}{p} \equiv \frac{\Theta_{p-1}}{p(p-1)} \bmod p = 7. 
	\]
	In particular, $U_{p^2}(\Theta_0'/p)$ is the reduction mod $p$ of an element $g_1 \in S_{3/2+6}(\Gamma_0(4\cdot 91), \Z)$. We can then verify experimentally using Magma that the first $421$ Fourier coefficients of $g_1$ are congruent modulo $p$ to the first $421$ Fourier coefficients of a modular form $g_2 \in S_{3/2}(\Gamma_0(4\cdot 91), \Z)$.
	
	We claim that this implies $g_1 \equiv g_2 \bmod p$. Indeed, let $\tilde{g_2} \in S_{3/2+6}(\Gamma_0(4\cdot 91), \Z)$ be such that $g_2 \equiv \tilde{g_2} \bmod p$. Then, the modular form $g_1 - \tilde{g_2} \in S_{3/2 + 6}(\Gamma_0(4\cdot 91), \Z)$ has the first $421$ Fourier coefficients equal to $0$ modulo $p$. This implies that the first $4 \cdot 421$ Fourier coefficients of $(g_1 - \tilde{g_2})^4 \in S_{30}(\Gamma_0(4\cdot 91), \Z)$ are congruent to $0$ modulo $p$. Since 
	\[
	421 \cdot 4 > \frac{30 \cdot [\SL_2(\Z): \Gamma_0(4\cdot 91)] }{12} = 1680,
	\]
	it follows from the Sturm bound (see \cite[Theorem 1]{St}) that $g_1 - \tilde{g_2} \equiv 0 \bmod p$, implying the desired result.
\end{proof}

Using a basis of $S_{3/2}(\Gamma_0(4\cdot91), \Z)$ given by Magma,  and using Proposition \ref{prop: Up^2 ofTheta0'/p is equal to a form of wt 3/2 mod p}, we can then compute $U_{p^2}^2(\Theta_0'/p)$ and verify the following: 
\begin{enumerate} 
	\item $\frac{1}{2}({U}_{p^2} + {U}_{p^2}^2)(\Theta_0'/p) \bmod p $ is an eigenvector for $U_{p^2}$ of eigenvalue $1$.
	\item $\frac{1}{2}({U}_{p^2} - {U}_{p^2}^2)(\Theta_0'/p) \bmod p$ is an eigenvector for $U_{p^2}$ of eigenvalue $-1$.
\end{enumerate}
It follows from there that, modulo $p$,
\begin{equation*}
	\begin{split}
		e_{\mathrm{ord}}\left( \frac{\Theta_0'}{p} \right) & =  \lim_{n \to +\infty} {U}_{p^2}^{n!} \frac{\Theta_0'}{p} = \lim_{n \to +\infty} {U}_{p^2}^{n! - 1} {U}_{p^2}\left(\frac{\Theta_0'}{p}\right) \\ & = \lim_{n \to +\infty} {U}_{p^2}^{n! - 1} \left( \frac{1}{2}({U}_{p^2} + {U}_{p^2}^2)(\Theta_0'/p) + \frac{1}{2}({U}_{p^2} - {U}_{p^2}^2)(\Theta_0'/p)  \right) \\ & = \frac{1}{2}({U}_{p^2} + {U}_{p^2}^2)(\Theta_0'/p) - \frac{1}{2}({U}_{p^2} - {U}_{p^2}^2)(\Theta_0'/p)  = {U}_{p^2}^2\left({\Theta_0'}/{p} \right).
	\end{split}
\end{equation*}
Based on this decomposition of $e_{\mathrm{ord}}\left( {\Theta_0'}/{p} \right)$, we will write 
\[
\mathrm{pr}_1(e_{\mathrm{ord}}(\Theta_0'/p)) =  ({U}_{p^2}(\Theta_0'/p) + {U}_{p^2}^2(\Theta_0'/p))/2,
\]
\[
\mathrm{pr}_{-1}(e_{\mathrm{ord}}(\Theta_0'/p)) =  -({U}_{p^2}(\Theta_0'/p) - {U}_{p^2}^2(\Theta_0'/p))/2.
\]

The results of the calculation are summarized in the following table.
\begin{table}[H]\label{table: Calculation e_ordTheta_0'/p}
	\begin{adjustwidth}{-2cm}{-2cm}
		\begin{center}
			\begin{tabular}{ ||c|y|g|g|o|y|y|g|y|y|g|g|o|g|g|o|| }
				\hline
				\multicolumn{1}{||c|}{Modular form} & \multicolumn{15}{c ||}{$q$-expansion} \\
				\multicolumn{1}{|| c | }{mod $p = 7$} & 
				2 & 5 & 6& 7& 8& 11& 13& 15& 18& 19& 20& 21& 24& 26 & 28 \\
				\hline
				$\Theta_0'/p$   & 2 & 3 & 2 & 4 & 5 & 4 & 3 & 3 & 2 & 0 & 3 & 6 & 0 & 3 & 4 \\
				${U}_{p^2}(\Theta_0'/p)$ & 3 & 3 & 5 & 2 & 0 & 3 & 6 & 3 & 5 & 1 & 2 & 2 & 2 & 3 & 1\\
				${U}_{p^2}^2(\Theta_0'/p)$& 3 & 4 & 2 & 4 & 0 & 3 & 1 & 3 & 5 & 6 & 5 & 6 & 5 & 4 & 4\\
				$({U}_{p^2}(\Theta_0'/p) + {U}_{p^2}^2(\Theta_0'/p))/2 $ & 3 & 0 & 0 & 3 & 0 & 3 & 0 & 3 & 5 & 0 & 0 & 4 & 0 & 0 & 6\\
				$({U}_{p^2}(\Theta_0'/p) - {U}_{p^2}^2(\Theta_0'/p))/2 $ & 0 & 3 & 5 & 6 & 0 & 0 & 1 & 0 & 0 & 1 & 2 & 5 & 2 & 3 & 2\\
				$e_{\mathrm{ord}}\Theta_0'/p$ &  3 & 4 & 2 & 4 & 0 & 3 & 1 & 3 & 5 & 6 & 5 & 6 & 5 & 4 & 4\\
				\hline
			\end{tabular}
			\caption{$D$th Fourier coefficients of linear combinations of ${U}_{p^2}^n(\Theta_0'/p)$ for $D$ such that $\left(\frac{-D}{13} \right) \neq 1$. For every $D$, we consider the color code blue: $\left(\frac{-D}{p} \right) = -1$, grey: $\left(\frac{-D}{p} \right) = 1$, red: $\left(\frac{-D}{p} \right) = 0$. }
		\end{center}
	\end{adjustwidth}
\end{table}

\begin{remark}
	In the decomposition $e_{\mathrm{ord}}(\Theta_0'/p) = \mathrm{pr}_1(e_{\mathrm{ord}}(\Theta_0'/p)) + \mathrm{pr}_{-1}(e_{\mathrm{ord}}(\Theta_0'/p))$ both summands are non-zero. The first summand is related to the Gross--Kohnen--Zagier generating series, as we proved in Theorem \ref{thm: log(GPhi(gamma)) = pr1(eordTheta_0')}. It would be interesting to find an arithmetic interpretation of the second summand, namely $\mathrm{pr}_{-1}(e_{\mathrm{ord}}(\Theta_0'/p))$.
\end{remark}

\subsection{Shimura lift and Hecke equivariance}\label{subsec: Hecke eigenvalues and Shimura lift}
The space $S_2^{\mathrm{new}}(\Gamma_0(7 \cdot 13))$ has dimension $7$, and there is a unique (up to scalars) cuspidal form such that $U_7$ acts by $1$ and has odd analytic rank (so in particular, the Hecke operator $U_{13}$ acts also by $1$). Its Fourier expansion is given by 
\begin{equation*}
	\begin{split}
		f & = q - 2q^3 - 2q^4 - 3q^5 + q^7 + q^9 + 4q^{12} + q^{13} + 6q^{15} + 4q^{16} - 6q^{17} - 7q^{19} + \dots \\
		& \equiv q  + 5q^3 + 5q^4 + 4q^5 + q^7 + q^9 + 4q^{12} + q^{13} + 6q^{15} + 4q^{16} + q^{17}  \dots \bmod p=7.
	\end{split}
\end{equation*}
Recall the Shimura lift 
\[
\sS_D \defeq \sS_{D, 0, 91}\colon S_{3/2}(\Gamma_0(4\cdot 91), \Q) \too S_{2}(\Gamma_0(2\cdot 91),\Q)
\]
defined in Section \ref{subsec: Shimura correspondence}, where $D$ is a square-free integer such that $-D > 0$. We computed the Shimura lift of $e_{\mathrm{ord}}(\Theta_0'/p)$ for different values of $D$ and obtained the following identities modulo $p$
\begin{equation}\label{eq: computation Shimura lift of eordTheta0'}
	\begin{split}
		&\sS_{-2}\left(\mathrm{pr}_1(e_{\mathrm{ord}}(\Theta_0'/p))\right) \equiv 6f \bmod p,\\
		&\sS_{-11}\left(\mathrm{pr}_1(e_{\mathrm{ord}}(\Theta_0'/p))\right) \equiv 3f + U_2f \bmod p,\\
		&\sS_{-15}\left(\mathrm{pr}_1(e_{\mathrm{ord}}(\Theta_0'/p))\right) \equiv 3f + 6(U_2f) \bmod p.\\
	\end{split}
\end{equation}
In particular, we see that $\mathrm{pr}_1e_{\mathrm{ord}}(\Theta_0'/p)$ is a Hecke eigenvector (mod $p$) with Hecke eigenvalues congruent to those of $f$.

The Schwartz--Bruhat function $\Phi$ is convenient. Indeed, since $\Phi$ is the difference of characteristic functions of the trace zero elements of two maximal orders, we deduce that $\Delta_\Phi(D)$ is of degree $0$ for every $D$. Moreover, $\deg_{\mathcal T_0}(\Phi)$ lands in the subspace of $\mathrm{Funct}(\Gamma \backslash \mathcal T_0, \Z)$ corresponding to weight two cusp forms of level $\Gamma_0(13 \cdot 3)$ that are old at $3$. Since $S_2(\Gamma_0(13), \Q) = 0$, we deduce $\deg_{\mathcal T_0}(\Phi) = 0$ implying the desired claim by the proof of Lemma \ref{lemma: div strong degree 0 and Hecke operators}.
Applying Theorem \ref{thm: log(GPhi(gamma)) = pr1(eordTheta_0')} to the Schwartz--Bruhat function $\Phi$ one obtains the equality
\[
\log_\gamma(G^+_\Phi)(q) = 2\mathrm{pr}_{1}(e_{\rm ord}(\Theta_0')).
\]
In particular, $(1/p)\log_p({G}^+_\Phi(\gamma)) \bmod p$ is a Hecke cuspidal eigenform of weight $3/2$ with the same Hecke eigenvalues as the cusp form $f$ of weight $2$ and level $\Gamma_0(91)$. 

On the other hand, consider $G_\Phi(q) \in J(\Q_{p^2})_\Q[[q]]$ and observe: 
\begin{itemize}
	\item The classes $[\Delta_\Phi(D)]$ are invariant under the action of $R^\times$ for every $D \in \cD_S$.
	\item The projection of the class $[\Delta_\Phi(D)]$ to a Hecke eigenspace is non-zero only if the eigenspace corresponds to an eigenform of rank $1$ by the Gross--Zagier formula.
	\item Via Jacquet--Langlands the Hecke action of $\T^N$ on $J(\Q_{p^2})$ factors through the action on $S_2^{91-\mathrm{new}}(\Gamma_0(91 \cdot 3) , \Q_{p^2})$.
	\item Since the divisors $\Delta_\Phi(D)$ on $X$ are obtained via pullback from divisors of a Shimura curve $\tilde{X}$ that is $p$-adically uniformized by $\tilde{\Gamma} \backslash \sH_p$, with $\tilde{\Gamma}$ the norm $1$ units of $\tilde{R}$, it follows that the classes $[\Delta_\Phi(D)]$ belong to the subspace corresponding to forms that are old at $3$. 
\end{itemize}
Hence, the functionals $\varphi\colon  J(\Q_{p^2}) \too \Q_{p^2}$ such that $\varphi(G_\Phi(q))$ is non-zero are generated by projections to eigenspaces where $\T^N$ acts with the same eigenvalues as it acts on eigenforms on $S_2^{\mathrm{new}}(\Gamma_0(91))$ which have rank $1$ and $U_7 = 1$. As we discussed above, there is a unique (up to scaling) such eigenform in $S_2^{\mathrm{new}}(\Gamma_0(91))$, which is $f$.
Uniqueness implies that $G_\Phi(q) \in J(\Q_{p^2})_{\Q}[[q]]$ is a non-zero multiple of $\log_\gamma(G_\Phi^+)(q)$, which has the same Hecke eigenvalues of $f$ modulo $p$. Hence, the calculation we presented gives an example (modulo $p$) of the Hecke equivariance property of the geometric theta lift provided by the Gross--Kohnen--Zagier generating series.

\bibliographystyle{abbrv}
\bibliography{bibfile}

\end{document}